\documentclass{article}
\usepackage[utf8]{inputenc}
\usepackage[a4paper, margin=3cm]{geometry}
\usepackage{mathtools}
\usepackage{amsmath}

\usepackage{amsthm}
\usepackage{dsfont}
\usepackage{amsfonts}
\usepackage{amssymb}
\usepackage[dvipsnames]{xcolor}

\newtheorem{theorem}{Theorem}[section]

\newtheorem{definition}[theorem]{Definition}

\newtheorem{lemma}[theorem]{Lemma}

\newtheorem{proposition}[theorem]{Proposition}

\newcommand{\E}[0]{\mathbb{E}}
\renewcommand{\P}[0]{\mathbb{P}}

\newcommand{\bbone}{\text{\usefont{U}{bbold}{m}{n}1}}
\MakeRobust{\bbone}

\newcommand{\ceil}[1]{\left\lceil #1 \right\rceil}

\DeclarePairedDelimiterX{\expectarg}[1]{[}{]}{%
  \ifnum\currentgrouptype=16 \else\begingroup\fi
  \activatebar#1
  \ifnum\currentgrouptype=16 \else\endgroup\fi
}

\usepackage{hyperref}
\hypersetup{
  colorlinks=true,
  linkcolor=NavyBlue,    
  citecolor=NavyBlue,    
  urlcolor=black,
  pdftitle={A simpler path to Ergodic Theorems for the Frontier of Branching Brownian Motion},
  pdfauthor={Gabriel Flath},
  pdfstartview=FitH,
  bookmarksnumbered=true
}

\title{A simpler path to Ergodic Theorems for the Frontier of Branching Brownian Motion}
\author{Gabriel Flath\thanks{Department of Statistics, University of Oxford, gabriel.flath@stats.ox.ac.uk}}

\begin{document}
\maketitle
\begin{abstract}
We revisit the ergodic theorem for the frontier of branching Brownian motion (BBM). Motivated by the proof of Arguin, Bovier, and Kistler \cite{arguin2012ergodic}, we provide a shorter and more direct argument. It relies on two observations: pairs of extremal particles observed at well-separated times must have branched early, and pairs of early-branching extremal particles have negatively correlated positions. This yields the ergodic theorem for BBM and extends it to a broad class of functionals of the recentred maximum. We also address a gap in the path localization argument of \cite{arguin2012ergodic}.

\end{abstract}


\section{Introduction and results}
The (standard binary) Branching Brownian motion (BBM) is a fundamental stochastic process describing the evolution of particles that move according to independent Brownian motions and undergo (binary) branching. It occupies a central position in probability theory and has deep connections with population models, reaction–diffusion equations, and disordered systems. The process starts from a single particle at the origin performing standard Brownian motion; each particle lives for an exponentially distributed time (with mean one), then splits into two offspring located at its current position. The descendants evolve independently according to the same dynamics.
\\

More formally, let $\mathcal{N}_t$ be the collection of particles at time $t$. For $v \in \mathcal{N}_t$ and $s\in \left[0,t\right]$, we say that $u\in \mathcal{N}_s$ is its ancestor at time $s$ if $v$ is a descendant of $u$ and denote it $u\preceq v$. For $u\in \mathcal{N}_t$ and $s\in \left[0,t\right]$, we write $X_u(s) \in \mathbb{R}$  for the position of particle $u$ (or its ancestor) at time $s$. We denote by $M_t := \max_{u\in \mathcal{N}_t} X_u(t)$ the position of the rightmost particle at time $t$. We define the branching Brownian motion process to be $X(t)=(X_u(t), u \in \mathcal{N}_t)$ and its natural filtration $\mathcal{F}_t=\sigma\{X(s), s\leq t \}$.\\

In a series of seminal works, Bramson~\cite{Bramson1983ConvergenceOS} showed that the position $M_t$ of the rightmost particle in a BBM is centred around $m_t := \sqrt 2 t -\frac3{2\sqrt 2 }\log t$, and more precisely that 
\[
M_t - m_t \xrightarrow[t\rightarrow \infty]{d} W,
\]
where $W$ is a random variable whose distribution function is given by the so-called critical travelling wave solution to the Fisher--KPP equation. Lalley and Sellke \cite{Lalley1987ACL} showed that the corresponding pointwise ergodic convergence
 cannot hold, i.e., that for $x\in \mathbb{R}$, $$
F_T(x) := \frac{1}{T} \int_0^T \bbone\{ M_t -m_t\leq x \} \, dt  \not \to \P(W\leq x) \text{ as } T \to \infty \text{ almost surely}. $$

They also proved that there exists a constant $C_\star>0$ such that
\begin{equation}\label{eq:LS}
     \P( M_t - m_t\le x) \to \E\left[ e^{-C_\star Z_\infty e^{-\sqrt 2 x}} \right],
\end{equation} 
where  $Z_\infty>0$ is the almost sure limit of the derivative martingale:
\begin{equation*}
    Z_t=\sum_{u \in \mathcal{N}_t}(\sqrt{2}t - X_u(t))e^{\sqrt{2}(X_u(t)-\sqrt{2}t)}.
\end{equation*}
 They further conjectured that this convergence holds in an ergodic sense. This was proved by Arguin, Bovier, and Kistler \cite{arguin2012ergodic}.

\begin{theorem}[Ergodic Theorem for $M_t$, Theorem 1 in \cite{arguin2012ergodic}]\label{thm:ergodic}
For any $x \in \mathbb{R}$,
\[
\lim_{T \to \infty}  F_T(x) = \exp\left(-C_\star Z_\infty e^{-\sqrt{2} x}\right), \quad \text{a.s.}
\]
\end{theorem}

The ergodic theorem extends to a broad class of functionals of the recentred maximum.

\begin{theorem}[Functional ergodic theorem]\label{thm:functional_ergodic}
Let $f:\mathbb{R}\to\mathbb{R}$ be measurable, with discontinuity set
$D_f$ of Lebesgue measure zero.
Assume that there exist constants $A>0$, $\alpha<\sqrt{2}$,
and $\beta<\sqrt{2}$ such that
\[
|f(x)|\le Ae^{\alpha x},\quad x\ge 0,
\qquad\text{and}\qquad
|f(x)|\le A\exp\!\big(e^{-\beta x}\big),\quad x\le 0.
\]
Then, almost surely,
\[
\frac{1}{T}\int_0^T f(M_t-m_t)\,dt
\;\xrightarrow[T\to\infty]{}\;
\E\!\left[
f\!\left(\frac{\mathcal{G}+\log C_\star+\log Z_\infty}{\sqrt{2}}\right)
\,\middle|\,Z_\infty\right],
\]
where $\mathcal{G}$ is a standard Gumbel random variable
independent of~$Z_\infty$, and $C_\star>0$ is the constant
in~\eqref{eq:LS}.
\end{theorem}

We give a shorter and more direct proof of Theorem~\ref{thm:ergodic}. The negative correlation inequality at the core of the argument (Proposition~\ref{thm:decorrelation}) applies beyond indicator functions, yielding Theorem~\ref{thm:functional_ergodic}. The argument distinguishes whether pairs of particles branch early, which along the way yields results on the genealogy of extremal particles that are of independent interest.

For $(u,v)\in \mathcal N_s\times\mathcal N_t$ we define their splitting time
\[
    Q(u,v) := \sup\{\gamma \ge 0 : X_u(\phi)=X_v(\phi)\ \text{for all } \phi \leq \gamma\}.
\]

For $x\in \mathbb{R}$ and $t\ge 0$, set 
$\mathcal N_t(x):=\{u\in\mathcal N_t:\ X_u(t)-m_t\ge x\}$, the set of particles above $m_t +x$ at time $t$.

The next result quantifies the fact that pairs of extremal particles observed at well-separated times must have branched early.

\begin{theorem}[Early branching across timescales]\label{thm:earlybr}
Let {$x\in\mathbb{R}$ and} $0\leq\rho<1/(2\sqrt2)$. For every $\varepsilon\in(0,\,1/2-\sqrt2\rho)$,
there exists $\ell>0$ such that for all $R\ge \ell$ and all $0<s<t$ with
$s\ge \ell$ and $t-s\ge \ell$,
\begin{equation}\label{eq:EB-moving}
\mathbb P\!\left(\exists (u,v)\in \mathcal N_s({x-\rho\log m})\times \mathcal N_t({x-\rho\log m}):
Q(u,v)\ge R\right)
\le m^{-1/2+\sqrt2\rho+\varepsilon},
\end{equation}
where $m:=\min\{t-s,\,s,\,R\}$.
\end{theorem}


\section{Proof of Theorems~\ref{thm:ergodic} and~\ref{thm:earlybr}}\label{proofff}
In this section, we establish Theorems~\ref{thm:ergodic} and~\ref{thm:earlybr}, proceeding via several intermediate results concerning the path localisation and the genealogy of the particles. The proof of Theorem~\ref{thm:functional_ergodic} is deferred to Section~\ref{sec:functional-proof}. The proofs of some intermediate results are deferred to Section~\ref{Sec:proofprop}.

We now introduce a path localisation and state two intermediate lemmas before turning to the proof of Theorem~\ref{thm:earlybr}.
\begin{definition}[$\alpha$-$t$-localised particle]
Fix $0<\alpha<1/2$. Let $t>0$ and $r\in[0,t/2]$. 
A particle $u \in \mathcal{N}_t$ is said to be \emph{$\alpha$-$t$-localised }on $[r,t-r]$ if, for every $s\in[r,t-r]$,
\[
X_u(s) \;\le\; \frac{s}{t}\,m_t \;-\; \min\{s,\,t-s\}^{\alpha}.
\]
Otherwise, $u$ is said to be \emph{not localised} on $[r,t-r]$.
\end{definition}

The following lemma is a path localisation estimate, in the spirit of Theorem~2.3 in \cite{Arguin2016}, but stated here with an explicit bound. The proof relies on an entropic repulsion argument implemented with two time scales, following the approach of Proposition 4.1 in \cite{flath1}. This yields the following bound.

\begin{lemma}\label{gih}
Let $0<\alpha< 1/2${, $x\in\mathbb{R}$,} and $0\leq\rho<1/\sqrt{2}$.
For any $\delta>0$, there exists $r_0$ such that, for all $r\ge r_0$ and $t\ge 3r$,
\begin{equation}\label{ojer}
\begin{split}
\mathbb{P}\!\Bigl(& \exists\, u \in \mathcal{N}_t({x-\rho\log r}) \text{ not $\alpha$-$t$-localised on } [\,r,\,t-r\,]\Bigr)
\;\le\; r^{\alpha +\sqrt{2}\rho - \tfrac12 + \delta}.
\end{split}
\end{equation}
\end{lemma}
The proof of Lemma~\ref{gih} is deferred to Section~\ref{Sec:proofprop}.

The following lemma quantifies that localised extremal particles at distant timescales cannot have branched too late. Its proof is deferred to Section~\ref{Sec:proofprop}.
\begin{lemma}\label{lemma-earlybrloc}
Let $0<\alpha<1/2${, $x\in\mathbb{R}$,} and let $\rho\geq 0$. There exist constants $C=C({x,}\rho,\alpha)$ and $r_0$ such that, for all $r\ge r_0$, $R\geq2r$, $s\ge3r$ and $t\ge s+2r$,
\begin{equation}\label{eq:path-localization-br}
\begin{split}
\mathbb{P}\!\Bigl(&\exists(u,v)\in\mathcal N_s({x-\rho \log r})\times\mathcal N_t({x-\rho \log r})\colon
u \text{ is $\alpha$--$s$--localised on }[r,s-r], \\
&\quad v \text{ is $\alpha$--$t$--localised on }[r,t-r],
\text{ and } Q(u,v)\ge R\Bigr) \\
&
\;\le\; Cr^{2\sqrt2\rho+\frac12}\left(re^{-\sqrt{2}R^\alpha} +(t-s)e^{-\sqrt{2}(t-s)^\alpha}\right).
\end{split}
\end{equation}
\end{lemma}

We now combine Lemmas~\ref{gih} and \ref{lemma-earlybrloc} to prove Theorem~\ref{thm:earlybr}.

\begin{proof}
Fix $\varepsilon\in(0,1/2-\sqrt2\rho)$. Choose $\alpha\in(0,\varepsilon/2)$ and set $\delta:=\varepsilon/2-\alpha>0$. For $0<s<t$ and $R>0$, define
\[
m:=\min\{t-s,\ s,\ R\},\qquad r:=m/3.
\]
For $m\ge \ell$ (with $\ell$ fixed below), the conditions of Lemmas~\ref{gih} and~\ref{lemma-earlybrloc} are satisfied: $r\ge r_0$, $s\ge 3r,\ R\ge 2r,\ t\ge s+2r$.
{Lemmas~\ref{gih} and~\ref{lemma-earlybrloc} are applied with $x-\rho\log 3$ in place of~$x$, so that the threshold $(x-\rho\log 3)-\rho\log r = x-\rho\log m$ matches the event in~\eqref{eq:EB-moving}.}

Let $E$ be the event in \eqref{eq:EB-moving}. Write
\[
A_s:=\mathbb P\!\left(\exists u\in\mathcal N_s({x-\rho\log m})\ \text{not $\alpha$--$s$--localised on }[r,s-r]\right),
\]
and define $A_t$ similarly. Let
\[
B:=\mathbb P\!\left(\exists (u,v)\in\mathcal N_s({x-\rho\log m})\times \mathcal N_t({x-\rho\log m}):\ u,v\ \text{are localised, and }Q(u,v)\ge R\right)
\]
be the probability in \eqref{eq:path-localization-br}.
Then $\mathbb P(E)\le A_s+A_t+B$.

By Lemma~\ref{gih},
\[
A_s+A_t\le 2\,r^{\alpha+\sqrt2\rho-\frac12+\delta}
\le C_0\,m^{-1/2+\sqrt2\rho+\varepsilon/2}.
\]

By Lemma~\ref{lemma-earlybrloc} and since $m\le R,t-s$, there exists $\ell_1$ such that for $m\ge \ell_1$,
\[
B\le C\,r^{2\sqrt2\rho+\frac12}
\left(re^{-\sqrt2 R^\alpha}
+(t-s)e^{-\sqrt2(t-s)^\alpha}\right)\le
C\,m^{-1/2+\sqrt2\rho+\varepsilon/2}.
\]

Thus, for $m$ large enough (say $m\ge \ell:=\max\{3r_0,\ell_1,(C_0+C)^{2/\varepsilon}\}$),
\[
\mathbb P(E)\le (C_0+C)\,m^{-1/2+\sqrt2\rho+\varepsilon/2}
\le m^{-1/2+\sqrt2\rho+\varepsilon},
\]
which establishes \eqref{eq:EB-moving}.
\end{proof}

We now turn to the proof of Theorem \ref{thm:ergodic}.
We first recall and follow the strategy of Arguin, Bovier, and Kistler \cite{arguin2012ergodic}. 
Let $x\in \mathbb{R}$, we introduce a cutoff parameter $\varepsilon>0$ and split the time integral at $\varepsilon T$: 
\begin{equation}\label{eq:split}
F_T(x)
= \frac{1}{T}\int_{\varepsilon T}^T \bbone\{M_t-m_t\le x\}\,dt
+ \frac{1}{T}\int_0^{\varepsilon T} \bbone\{M_t-m_t\le x\}\,dt.
\end{equation}
The second term does not contribute in the limit as $T \to \infty$ first and $\varepsilon \to 0$ next. It therefore suffices to compute the double limit for the first term. 

To this end, let $R_T>0$ and decompose the empirical distribution on $[\varepsilon T, T]$ as
\begin{align}\label{eq:empdist}
\frac{1}{T}\int_{\varepsilon T}^T \bbone\{M_t - m_t \le x\}&\,dt
= \frac{1}{T}\int_{\varepsilon T}^T 
   \mathbb{P}\!\left(M_t - m_t \le x \,\middle|\, \mathcal{F}_{R_T}\right)\,dt \\
&+ \frac{1}{T}\int_{\varepsilon T}^T 
   \Bigl( \bbone\{M_t - m_t \le x\}
   - \mathbb{P}\!\left(M_t - m_t \le x \,\middle|\, \mathcal{F}_{R_T}\right)\Bigr)\,dt. \nonumber
\end{align}

Theorem 2 in \cite{arguin2012ergodic} established that,
for $R_T \to \infty$ as $T \to \infty$ with $R_T = o(\sqrt{T})$,
\begin{equation}\label{eq:cond-mean-conv}
\lim_{\varepsilon \to 0}\;
\lim_{T \to \infty}\;
\frac{1}{T}\int_{\varepsilon T}^T
\mathbb{P}\!\left(M_t - m_t \le x \,\middle|\, \mathcal{F}_{R_T}\right)\,dt
= \exp\left(-C_\star Z_\infty e^{-\sqrt{2} x}\right),
\quad \text{a.s.}
\end{equation}
 
Consequently, by Theorem~2 of \cite{arguin2012ergodic}, to establish Theorem~\ref{thm:ergodic} it remains only to show that the second term on the right-hand side of \eqref{eq:empdist} converges to zero almost surely. This is the content of the next theorem, whose proof occupies the remainder of this section.

\begin{theorem}\label{thm:fluctuations}
Fix $x\in \mathbb{R}$, $\varepsilon > 0$ and $\delta > 0$, and define $R_T = \log(T)^{2+\delta}$. Then
\[
\lim_{T \to \infty} 
\frac{1}{T}\int_{\varepsilon T}^T 
\Bigl( \bbone\{M_t - m_t \ge x\}
      - \mathbb{P}\!\left(M_t - m_t \ge x \,\middle|\, \mathcal{F}_{R_T}\right)\Bigr)\,dt
= 0, \quad \text{a.s.}
\]
\end{theorem}
Note that a similar statement appears as Theorem~3 in~\cite{arguin2012ergodic}. However, as we explain in Section~\ref{sec:discussion}, there is a gap in their proof, which is resolved by the alternative argument we provide here. The key step is the following result.
\begin{proposition}\label{thm:decorrelation}
Let $0\leq R\leq s \leq t$, let $x\in \mathbb{R}$ and $y\in \mathbb{R}$. Let $u_s^*$ (resp.~$u_t^*$) denote the rightmost particle at time $s$ (resp.~$t$). Then 
        \begin{equation}\label{eq:deco}
    \begin{aligned}
        \mathbb{P}\!\left(M_s \geq x, \,M_t \geq y,\, Q(u^{*}_s,u^{*}_t)\leq R \,\middle| \mathcal{F}_{R}\right)\leq \mathbb{P}\!\left(M_s \geq x \,\middle| \mathcal{F}_{R}\right)\mathbb{P}\!\left(M_t \geq y \,\middle| \mathcal{F}_{R}\right).
    \end{aligned}
    \end{equation}
\end{proposition}
Before proving Proposition~\ref{thm:decorrelation}, we state the following corollary of the van den Berg--Kesten--Reimer inequality. Its proof is deferred to Section~\ref{Sec:proofprop}.
\begin{lemma}[BKR corollary for cross--index unions]\label{lem:BK-cross}  
Let \(U\) be an index set. 
For each \(u \in U\), let \((A_u,E_u)\) be a pair of events.
Assume that the family of sigma-algebras \(\{\sigma(A_u,E_u) : u \in U\}\) is mutually independent. 
Then 
\[
\mathbb{P}\!\left(\bigcup_{\substack{u,v \in U \\ u \neq v}} A_u \cap E_v \right)
\;\leq\;
\mathbb{P}\!\left(\bigcup_{u \in U} A_u \right)\,
\mathbb{P}\!\left(\bigcup_{u \in U} E_u \right).
\]
\end{lemma}  

\begin{proof}[Proof of Proposition \ref{thm:decorrelation}]
For $u \in \mathcal N_R$, define
\[
M^u_s \coloneqq \max_{v \in \mathcal N_s : u \preceq v} X_v(s)
\quad \text{(resp. } M^u_t\text{)}
\]
as the rightmost particle at time $s$ (resp. $t$) in the subtree rooted at $u$, and set
\[
A_u := \{ M^u_s \ge x \}, \qquad
E_u := \{ M^u_t  \ge y \}.
\]

On $\{Q(u_s^*,u_t^*)\le R\}$ the rightmost particles at times $s$ and $t$ have distinct ancestors at time $R$; hence
\[
\{M_s\geq x,\ M_t\geq y,\ Q(u_s^*,u_t^*)\le R\}
\subseteq \bigcup_{\substack{u,v\in\mathcal N_R\\ u\neq v}} \bigl(A_u\cap E_v\bigr)
\quad\text{a.s.}
\]

Condition on $\mathcal F_R$. By the branching property, the descendant trees below distinct $u\in\mathcal N_R$ are independent. Hence the family $\{(A_u,E_u)\}_{u\in\mathcal N_R}$ satisfies the hypothesis of Lemma~\ref{lem:BK-cross}. Applying that lemma yields
\[
\mathbb P\!\left(\bigcup_{\substack{u,v\in\mathcal N_R\\ u\neq v}} (A_u\cap E_v)\,\middle|\,\mathcal F_R\right)
\le
\mathbb P\!\left(\bigcup_{u\in\mathcal N_R}A_u\,\middle|\,\mathcal F_R\right)
\mathbb P\!\left(\bigcup_{v\in\mathcal N_R}E_v\,\middle|\,\mathcal F_R\right).
\]
Finally,
\[
\bigcup_{u\in\mathcal N_R}A_u=\{M_s\geq x\},\qquad
\bigcup_{v\in\mathcal N_R}E_v=\{M_t\geq y\},
\]
so the right-hand side equals
$\mathbb P(M_s\geq x\,|\,\mathcal F_R)\,
 \mathbb P(M_t\geq y\,|\,\mathcal F_R)$,
which establishes \eqref{eq:deco}, as claimed.
\end{proof}

We now prove Theorem \ref{thm:fluctuations}.
\begin{proof}[Proof of Theorem \ref{thm:fluctuations}]
Let $R:=R_T=(\log T)^{2+\delta}$ for some $\delta>0$. For $t\ge0$ set
\[
A_t:=\{M_t-m_t\geq x\},\qquad
Y_t:=\bbone_{A_t}-\mathbb{P}(A_t\mid\mathcal F_R),\qquad
\Delta_T:=\frac{1}{T}\int_{\varepsilon T}^T Y_t\,dt .
\]
The statement of Theorem \ref{thm:fluctuations} is that $\Delta_T\to0$ almost surely as $T\to \infty$.

Square and symmetrize, then split to yield
\[
\Delta_T^2=\frac{2}{T^2}\int_{\varepsilon T}^T\!\!\int_s^T Y_sY_t\,dt\,ds
=\underbrace{\frac{2}{T^2}\int_{\varepsilon T}^T\!\!\int_s^{(s+R)\wedge T} Y_sY_t\,dt\,ds}_{=:I_T}
+\underbrace{\frac{2}{T^2}\int_{\varepsilon T}^T\!\!\int_{(s+R)\wedge T}^{T} Y_sY_t\,dt\,ds}_{=:J_T}.
\]
Since $|Y_sY_t|\le 1$, $I_T \le 2R/T$, thus $\E[ I_T]\leq 2R/T$. We now bound $\E[J_T]$. For $R\le s\le t$, applying Proposition~\ref{thm:decorrelation} on $\{Q(u_s^*,u_t^*)\le R\}$,
\begin{equation*}
    \begin{split}
      \mathbb{E}[Y_sY_t\mid\mathcal F_R]
&\le \mathbb{P}\!\big(A_s\cap A_t\cap\{Q(u_s^*,u_t^*)\ge R\}\mid\mathcal F_R\big)\\
&\le  \P\!\left(\exists (u,v)\in \mathcal N_s(x)\times \mathcal N_t(x):\ Q(u,v)\ge R \mid\mathcal F_R \right).  
    \end{split}
\end{equation*}
Therefore, for any $\alpha \in (0,1/2)$, Theorem~\ref{thm:earlybr} {(applied with $\rho=0$)} gives, for sufficiently large $T$ (and therefore sufficiently large $R$),
\[
\mathbb{E}[J_T]
\le \frac{2}{T^2}\int_{\varepsilon T}^T\!\!\int_{s+R}^{T} \mathbb{E}[Y_sY_t]\,dt\,ds \le \frac{2}{T^2}\int_{\varepsilon T}^T\!\!\int_{s+R}^{T} R^{-\alpha}\,dt\,ds
\le R^{-\alpha}.
\]
Let $T_n=e^{n^\beta}$ with $\beta\in(0,1)$. Let $\eta>0$, by Markov’s inequality,
\[
\sum_n \mathbb{P}\big(\Delta_{T_n}^2>\eta\big) \le\eta^{-1}\sum_n \mathbb{E}[I_{T_n}]+\mathbb{E}[J_{T_n}] \le \eta^{-1}\sum_{n=1}^\infty \frac{2(\log T_n)^{2+\delta}}{T_n}
+(\log T_n)^{-(2+\delta)\alpha} <\ \infty,
\]
whenever
\[
\beta(2+\delta)\alpha>1,
\]
which is achievable for every $\delta>0$ by taking $\beta \in \big(\tfrac{2}{2+\delta}, 1\big)$ and $\alpha$ arbitrarily close to $\tfrac{1}{2}$. 
Hence, $\Delta_{T_n} \to 0$ almost surely. 

Combined with~\eqref{eq:cond-mean-conv}, this yields convergence of the time average of $\bbone_{A_t}$ along the subsequence. By Lemma~\ref{subtoful} below, applied to the fixed integrand $\bbone_{A_t}$, convergence extends to all $T$, and therefore $\Delta_T\to 0$ almost surely. The proof of Lemma~\ref{subtoful} is deferred to Section~\ref{Sec:proofprop}.

\begin{lemma}[Subsequence-to-full-time convergence for time averages]\label{subtoful}
Let $x: [0, \infty) \to \mathbb{R}$ be a measurable function bounded from below by $M \in \mathbb{R}$. 
For $T>0$, define the time averages
\[
\rho_T := \frac{1}{T}\int_0^T x(t)\,dt .
\]
Let $(S_n)_{n\ge 1}$ be an increasing sequence with $S_n\to\infty$ and $S_{n+1}/S_n\to 1$. 
If
\[
\rho_{S_n} \xrightarrow[n\to\infty]{} \rho,
\]
then
\[
\rho_T \xrightarrow[T\to\infty]{} \rho .
\]
\end{lemma}

This concludes the proof of Theorem~\ref{thm:fluctuations}.
\end{proof}

\section{Proof of Theorem~\ref{thm:functional_ergodic}}\label{sec:functional-proof}

The proof of Theorem~\ref{thm:functional_ergodic} relies on the following conditional approximation, which extends Theorem~2 of~\cite{arguin2012ergodic} to growing thresholds.

\begin{proposition}[Growing-threshold conditional approximation]%
\label{prop:growing-threshold}
Fix $\delta\in(0,\tfrac{1}{4})$ and let
$R_T\to\infty$ with $R_T=o(T^{(1-\delta)/2})$.
Let $D=D(T)$ satisfy
\[
-c\,\log R_T
\;\le\; D
\;\le\; \frac{\sqrt{R_T}}{(\log R_T)^{2}}
\]
for some fixed $c<1/(2\sqrt{2})$.
Then, uniformly over $s\in[T^{1-\delta},\,T]$,
\begin{equation}\label{eq:growing}
\P\!\left(M_s-m_s\le D
  \,\middle|\,\mathcal{F}_{R_T}\right)
= \exp\!\bigl(-(1+o(1))\,C_\star\,Z_\infty\,e^{-\sqrt{2}\,D}\bigr),
\qquad\text{a.s.,}
\end{equation}
where $o(1)\to 0$ a.s.\ as $T\to\infty$.
\end{proposition}

\begin{proof}[Proof of Proposition~\ref{prop:growing-threshold}]
The argument follows~\cite[Theorem~2]{arguin2012ergodic}.
For each $u\in\mathcal{N}_{R_T}$, set
\[
y_u:=\sqrt{2}\,R_T-X_u(R_T),
\qquad
z_u:=y_u\,e^{-\sqrt{2}\,y_u},
\]
and write $Y:=Y(R_T)=\sum_u e^{-\sqrt{2}\,y_u}$ and
$Z:=Z_{R_T}=\sum_u z_u$.
Since $s\ge T^{1-\delta}$ and $R_T=o(T^{(1-\delta)/2})$,
one has $R_T=o(\sqrt{s}\,)$.

By the branching property at time~$R_T$
(cf.~\cite[(3.1)--(3.3)]{arguin2012ergodic}),
\begin{equation}\label{eq:prod}
\P\!\left(M_s-m_s\le D
  \,\middle|\,\mathcal{F}_{R_T}\right)
= \prod_{u\in\mathcal{N}_{R_T}}(1-a_u),
\end{equation}
where
\begin{gather*}
a_u:=\P\bigl(M_{s-R_T}-m_{s-R_T}>\xi_u+\epsilon_{s,T}\bigr),
\qquad
\xi_u:=D+y_u,\\
\epsilon_{s,T}:=\tfrac{3}{2\sqrt{2}}\,
   \log\!\Bigl(\frac{s-R_T}{s}\Bigr)
   \xrightarrow[T\to\infty]{} 0
   \quad\text{uniformly in $s$.}
\end{gather*}

We verify the hypotheses of the uniform right-tail
estimate~\cite[Corollary~1.2]{Chataignier2025}
with $t=s-R_T$.
By~\cite[Theorem~1.1]{Hu2012TheAS},
\begin{equation}\label{eq:miny}
\min_u y_u
\;\ge\; \frac{1}{2\sqrt{2}}\,\log R_T\,(1+o(1))
\qquad\text{a.s.\ as } R_T\to\infty.
\end{equation}
Since $c<1/(2\sqrt{2})$ and $D\ge -c\log R_T$,
it follows that
$\min_u \xi_u = D + \min_u y_u \to +\infty$
a.s.\ as $T\to\infty$.
Moreover,
\[
\max_u \xi_u
\;\le\; |D|+\max_u y_u
\;=\; O(R_T)
\;=\; o(\sqrt{s}\,),
\]
since $\max_u y_u=O(R_T)$ a.s.\
by~\cite{Bramson1983ConvergenceOS} and $|D|=o(R_T)$.
Therefore~\cite[Corollary~1.2]{Chataignier2025} gives,
a.s.\ and uniformly in~$u$,
\[
a_u=C_\star(1+o_T(1))\,\xi_u\,e^{-\sqrt{2}\,\xi_u},
\]
where $o_T(1)\to 0$ a.s.\ as $T\to\infty$.
Consequently,
\begin{equation}\label{eq:sumak}
\sum_u a_u
= C_\star(1{+}o_T(1))\,e^{-\sqrt{2}\,D}
  \sum_u \xi_u\,e^{-\sqrt{2}\,y_u}
= C_\star(1{+}o_T(1))\,e^{-\sqrt{2}\,D}\,(Z+D\,Y).
\end{equation}

By~\cite[Theorem~1.2]{Hu2012TheAS},
$Y\le(\log R_T)^{1+\varepsilon}/\!\sqrt{R_T}$
a.s.\ for all large~$T$ and every $\varepsilon>0$.
Since $|D|\le\sqrt{R_T}/(\log R_T)^2$ for all large~$T$,
\[
|D|\,Y
\;\le\; (\log R_T)^{\varepsilon-1}
\;\xrightarrow[T\to\infty]{}\;0
\qquad\text{a.s.,\;for }\varepsilon\in(0,1).
\]
Since $Z\xrightarrow[T\to\infty]{\text{a.s.}} Z_\infty>0$,
we conclude that $Z+D\,Y\to Z_\infty$ a.s.\ and hence
\begin{equation}\label{eq:mainterm}
\sum_u a_u = C_\star(1{+}o_T(1))\,Z_\infty\,e^{-\sqrt{2}\,D}
\qquad\text{a.s.}
\end{equation}

Since $\min_u \xi_u\to+\infty$ a.s., we have $\max_u a_u\to0$
a.s.\ as $T\to\infty$.
Using $-a-a^2\le\log(1-a)\le -a$ for $a\in(0,\tfrac{1}{2})$
and $\sum_u a_u^2\le\bigl(\max_u a_u\bigr)\bigl(\sum_u a_u\bigr)
=o(e^{-\sqrt{2}\,D})$ a.s.,
\[
\frac{\sum_u\log(1-a_u)}{e^{-\sqrt{2}\,D}}
\;\xrightarrow[T\to\infty]{\text{a.s.}}\;
-C_\star\,Z_\infty,
\]
which yields~\eqref{eq:growing}.
\end{proof}

\begin{proof}[Proof of Theorem~\ref{thm:functional_ergodic}]
Write $\hat{M}_t:=M_t-m_t$.
Fix $K>0$ with $\pm K\notin D_f$ and decompose
\[
f=f^{(K)}+f_K^++f_K^-,
\]
where
$f^{(K)}:=f\,\bbone_{[-K,K]}$,\;
$f_K^+:=f\,\bbone_{(K,\infty)}$,\;
$f_K^-:=f\,\bbone_{(-\infty,-K)}$.
Fix $\delta\in(0,\tfrac{1}{4})$, $\vartheta\in(1-\delta,1)$,
and $b>\max\{1,\beta\}$.
To control the tails, introduce growing truncation levels
\[
L_t^+:=q_+\log\log t,
\qquad
L_t^-:=q_-\log\log\log t,
\]
where $q_->1$ satisfies
\begin{equation}\label{eq:qminus_choice}
\beta q_-<b,
\end{equation}
and $q_+>0$ is chosen large enough to
satisfy~\eqref{eq:qplus_choice} below.
Set
\[
R_T:=\exp\!\big((\log\log T)^b\big).
\]
Since $L_t^-=o(\log R_T)$ and $L_t^+=o\!\big(\sqrt{R_T}/(\log R_T)^2\big)$ as $T\to\infty$, Proposition~\ref{prop:growing-threshold} applies with any threshold $D\in[-L_t^-,\,L_t^+]$, uniformly over $t\in[T^\vartheta,T]$, for all large~$T$.

\medskip
\noindent\textbf{Truncated part.}\;
Since $f^{(K)}$ is bounded with Lebesgue-null discontinuity set,
it can be uniformly approximated by step functions whose jumps
avoid~$D_f$.
Each step function is a linear combination of half-line
indicators, so Theorem~\ref{thm:ergodic} yields
\begin{equation}\label{eq:bounded_core}
\frac{1}{T}\int_0^T f^{(K)}(\hat{M}_t)\,dt
\;\xrightarrow[T\to\infty]{}\;
\E\!\left[
f^{(K)}\!\left(\frac{\mathcal{G}+\log C_\star+\log Z_\infty}{\sqrt{2}}\right)
\,\middle|\,Z_\infty\right]
\qquad\text{a.s.}
\end{equation}
It remains to show that both tails vanish as $K\to\infty$.

\medskip
\noindent\textbf{Right tail.}\;
We truncate $f_K^+$ at a growing level $L_t^+$ and control separately the conditional mean, the fluctuations, and the remainder beyond the truncation.
The growth bound $|f(x)|\le Ae^{\alpha x}$ for $x\ge0$ gives
\[
|f_K^+(x)|\le A\,\phi_t^+(x)+A\,h_t^+(x),
\]
where
\[
\phi_t^+(x):=e^{\alpha\min(x,\,L_t^+)}\,\bbone_{\{x\ge K\}},
\qquad
h_t^+(x):=e^{\alpha x}\,\bbone_{\{x>L_t^+\}}.
\]

On the interval $[0,T^\vartheta]$, the contribution
from~$\phi_t^+$ satisfies
\begin{equation}\label{eq:right_early}
\frac{1}{T}\int_0^{T^\vartheta}\phi_t^+(\hat{M}_t)\,dt
\;\le\; T^{\vartheta-1}\,e^{\alpha L_{T^\vartheta}^+}
\;\le\; T^{\vartheta-1}\,(\log T)^{\alpha q_+}
\;\xrightarrow[T\to\infty]{}\; 0.
\end{equation}

On the interval $[T^\vartheta,T]$, we bound the conditional mean
and the fluctuation separately.
By Proposition~\ref{prop:growing-threshold}, uniformly for
$t\in[T^\vartheta,T]$ and $y\in[K,L_t^+]$,
\[
\P\!\left(\hat{M}_t\ge y\,\middle|\,\mathcal{F}_{R_T}\right)
\;\le\; (1+o(1))\,C_\star Z_\infty\,e^{-\sqrt{2}\,y},
\]
where $o(1)\to0$ a.s.\ as $T\to\infty$.
The function~$\phi_t^+$ admits the Stieltjes representation
\[
\phi_t^+(x)
=\int_{[K,L_t^+]}\bbone_{\{x\ge y\}}\,\nu_t^+(dy),
\]
where
$\nu_t^+(dy):=e^{\alpha K}\delta_K(dy)
+\alpha e^{\alpha y}\bbone_{[K,L_t^+]}(y)\,dy$
has total mass $\nu_t^+([K,L_t^+])=e^{\alpha L_t^+}$.
Integrating the conditional tail bound against~$\nu_t^+$ yields
\begin{equation}\label{eq:right_cond_mean}
\sup_{t\in[T^\vartheta,T]}
\E\!\left[\phi_t^+(\hat{M}_t)\,\middle|\,\mathcal{F}_{R_T}\right]
\;\lesssim\; C_\star Z_\infty\,
e^{(\alpha-\sqrt{2})K}
\;\xrightarrow[K\to\infty]{}\;0
\qquad\text{a.s.}
\end{equation}

We now control the fluctuation around this conditional mean.
Set
\[
Y_t^{+}:=\phi_t^+(\hat{M}_t)
-\E\!\left[\phi_t^+(\hat{M}_t)\,\middle|\,\mathcal{F}_{R_T}\right],
\qquad
\overline{F}_T^{+,K}:=\frac{1}{T}\int_{T^\vartheta}^T Y_t^{+}\,dt.
\]
Squaring and symmetrising,
\begin{equation}\label{eq:right_var}
\E\!\left[\big(\overline{F}_T^{+,K}\big)^2\right]
\;\le\; \frac{2}{T^2}\int_{T^\vartheta}^T\!\!\int_s^{(s+R_T)\wedge T}
    \E[Y_s^+Y_t^+]\,dt\,ds
  \;+\; \frac{2}{T^2}\int_{T^\vartheta}^T\!\!\int_{(s+R_T)\wedge T}^{T}
    \E[Y_s^+Y_t^+]\,dt\,ds.
\end{equation}
Since $|Y_t^+|\le 2e^{\alpha L_t^+}\le 2(\log T)^{\alpha q_+}$,
the close-time contribution (first integral) is at most
$4R_T(\log T)^{2\alpha q_+}/T$.

For the well-separated contribution ($t-s\ge R_T$),
the Stieltjes representation gives
\[
Y_t^+=\int_{[K,L_t^+]}\xi_t^+(y)\,\nu_t^+(dy),
\qquad
\xi_t^+(y):=\bbone_{\{\hat{M}_t\ge y\}}
-\P\!\left(\hat{M}_t\ge y\,\middle|\,\mathcal{F}_{R_T}\right).
\]
For fixed $y,z\ge K$, the covariance
$\E[\xi_s^+(y)\,\xi_t^+(z)\mid\mathcal{F}_{R_T}]$
splits according to whether the rightmost particles at
times~$s$ and~$t$ share a common ancestor after
time~$R_T$.
By Proposition~\ref{thm:decorrelation}, the contribution
from pairs branching before time~$R_T$ is non-positive, so
\[
\E\!\left[\xi_s^+(y)\,\xi_t^+(z)
\,\middle|\,\mathcal{F}_{R_T}\right]
\;\le\; \P\!\left(\exists(u,v)\in\mathcal{N}_s(y)
  \times\mathcal{N}_t(z):\,
  Q(u,v)\ge R_T
\,\middle|\,\mathcal{F}_{R_T}\right).
\]
Since $y,z\ge K$,
$\mathcal{N}_s(y)\subseteq\mathcal{N}_s(K)$
and $\mathcal{N}_t(z)\subseteq\mathcal{N}_t(K)$, so
the right-hand side is bounded by
$\P\bigl(\exists(u,v)\in\mathcal{N}_s(K)
\times\mathcal{N}_t(K):\,Q(u,v)\ge R_T\,\big|\,\mathcal{F}_{R_T}\bigr)$.
By Theorem~\ref{thm:earlybr} (applied with {$x=K$ and} $\rho=0$),
for every $\varepsilon\in(0,\tfrac{1}{2})$ and all
sufficiently large~$T$,
\[
\E[Y_s^+Y_t^+]
\;\le\; (\log T)^{2\alpha q_+}\,R_T^{-1/2+\varepsilon}.
\]
Substituting into~\eqref{eq:right_var},
\begin{equation}\label{eq:right_var_bound}
\E\!\left[\big(\overline{F}_T^{+,K}\big)^2\right]
\;\le\; \frac{4R_T(\log T)^{2\alpha q_+}}{T}
  \;+\; (\log T)^{2\alpha q_+}\,R_T^{-1/2+\varepsilon}.
\end{equation}
Let $T_n:=\exp(n^\eta)$ with $\eta\in(0,1)$, so that
$T_{n+1}/T_n\to 1$ as $n\to\infty$.
Along this subsequence,
$R_{T_n}=\exp\!\big((\eta\log n)^b\big)$
with $b>1$, so the right-hand side
of~\eqref{eq:right_var_bound} is summable in~$n$.
By Markov's inequality and Borel--Cantelli,
\begin{equation}\label{eq:right_fluct_subseq}
\overline{F}_{T_n}^{+,K}\;\xrightarrow[n\to\infty]{\text{a.s.}}\; 0
\qquad\text{for each fixed~$K$.}
\end{equation}

It remains to handle the extreme right-tail remainder, beyond the
truncation level~$L_t^+$.
Choose $q_+>0$ large enough that
\begin{equation}\label{eq:qplus_choice}
\eta\,q_+\,(\sqrt{2}-\alpha)>1.
\end{equation}
By~\cite[Corollary~4.2]{Chataignier2025},
$\P(\hat{M}_t>x)\le C'(1+x_+)\,e^{-\sqrt{2}\,x}$
for all $t\ge2$ and $x\in\mathbb{R}$, so
\[
\E[h_t^+(\hat{M}_t)]
\;\le\; C''\,(\log\log t)\,(\log t)^{-q_+(\sqrt{2}-\alpha)}.
\]
Therefore
\[
\E\!\left[\frac{1}{T_n}\int_0^{T_n}h_t^+(\hat{M}_t)\,dt\right]
\;\lesssim\;(\log n)\,n^{-\eta q_+(\sqrt{2}-\alpha)},
\]
which is summable in~$n$ by~\eqref{eq:qplus_choice}.
Borel--Cantelli then gives
$\frac{1}{T_n}\int_0^{T_n}h_t^+(\hat{M}_t)\,dt\to0$ a.s.\
as $n\to\infty$.

Combining~\eqref{eq:right_early},~\eqref{eq:right_cond_mean},
\eqref{eq:right_fluct_subseq}, and the extreme right-tail bound, and
extending from the subsequence~$(T_n)$ to the full time axis
via Lemma~\ref{subtoful}, we obtain
\begin{equation}\label{eq:right_tail_done}
\lim_{K\to\infty}\limsup_{T\to\infty}
\frac{1}{T}\int_0^T|f_K^+(\hat{M}_t)|\,dt=0
\qquad\text{a.s.}
\end{equation}

\medskip
\noindent\textbf{Left tail.}\;
The argument for the left tail parallels the right tail,
with two modifications: the double-exponential growth
of~$f$ for large negative arguments requires a different
truncation, and the contribution from the extreme left
vanishes by an almost sure lower bound on~$\hat{M}_t$ rather
than by a moment bound.

Define
\[
\phi_t^-(x)
:=\exp\!\big(e^{\beta\min(-x,\,L_t^-)}\big)\,
  \bbone_{\{x\le-K\}}.
\]
By~\cite[Theorem~1.1]{Hubig},
\begin{equation}\label{eq:Hu_cutoff}
\hat{M}_t\;\ge\; -L_t^-
\qquad\text{for all sufficiently large }t,\;\text{a.s.}
\end{equation}
On the event $\{\hat{M}_t\ge-L_t^-\}$, one has
$\min(-\hat{M}_t,L_t^-)=-\hat{M}_t$, so
$|f_K^-(\hat{M}_t)|\le A\,\phi_t^-(\hat{M}_t)$.
In particular, $|f_K^-(\hat{M}_t)|\le A\,\phi_t^-(\hat{M}_t)$
for all $t\ge t_0(\omega)$, where $t_0$ is almost surely
finite.

On $[0,T^\vartheta]$, the bound
\[
\phi_t^-(\hat{M}_t)
\;\le\;\exp\!\big(e^{\beta L_{T^\vartheta}^-}\big)
\;=\;\exp\!\big(O((\log\log T)^{\beta q_-})\big)
\;=\;T^{o(1)}
\]
yields a contribution of order
$O(T^{\vartheta-1+o(1)})\to0$ as $T\to\infty$.

On $[T^\vartheta,T]$,
Proposition~\ref{prop:growing-threshold} gives the
conditional left-tail approximation:
uniformly for $y\in[K,L_t^-]$ and $t\in[T^\vartheta,T]$,
\[
\P\!\left(\hat{M}_t\le -y\,\middle|\,\mathcal{F}_{R_T}\right)
\;=\; \exp\!\bigl(-(1+o(1))\,C_\star Z_\infty\,e^{\sqrt{2}\,y}\bigr).
\]
The function~$\phi_t^-$ admits the Stieltjes representation
\[
\phi_t^-(x)
=\int_{[K,L_t^-]}\bbone_{\{x\le-y\}}\,\nu_t^-(dy),
\]
where
$\nu_t^-(dy):=\exp\!\big(e^{\beta K}\big)\,\delta_K(dy)
+\beta\,e^{\beta y}\exp\!\big(e^{\beta y}\big)\,
\bbone_{[K,L_t^-]}(y)\,dy$
has total mass
$\nu_t^-([K,L_t^-])
=\exp\!\big(e^{\beta L_t^-}\big)
=\exp\!\big((\log\log t)^{\beta q_-}\big)$.
Integrating the conditional tail against~$\nu_t^-$,
and using $\beta<\sqrt{2}$ to verify that the
double-exponential decay $\exp(-c\,e^{\sqrt{2}\,y})$
dominates the growth $\exp(e^{\beta y})$ as $y\to\infty$,
we obtain
\begin{equation}\label{eq:left_cond_mean}
\sup_{t\in[T^\vartheta,T]}
\E\!\left[\phi_t^-(\hat{M}_t)\,\middle|\,\mathcal{F}_{R_T}\right]
\;\longrightarrow\;0
\qquad\text{as $K\to\infty$, a.s.}
\end{equation}

For the fluctuations, set
$Y_t^-:=\phi_t^-(\hat{M}_t)-\E\!\big[\phi_t^-(\hat{M}_t)\,\big|\,\mathcal{F}_{R_T}\big]$
and
$\overline{F}_T^{-,K}:=\frac{1}{T}\int_{T^\vartheta}^T Y_t^-\,dt$.
Writing
\[
\xi_t^-(y)
:=\bbone_{\{\hat{M}_t\le-y\}}
-\P\!\left(\hat{M}_t\le-y\,\middle|\,\mathcal{F}_{R_T}\right)
\;=\; -\bigl(\bbone_{\{\hat{M}_t>-y\}}
  -\P\!\left(\hat{M}_t>-y\,\middle|\,\mathcal{F}_{R_T}\right)\bigr),
\]
Proposition~\ref{thm:decorrelation} applies to the
complementary events $\{\hat{M}_t>-y\}$, so the early-branching
contribution remains non-positive.
Fix $\rho\in(0,(2\sqrt{2})^{-1})$.
Since $L_t^-\le\rho\log R_T$ for all large~$T$,
the covariance bound reduces to
\[
\E\!\left[\xi_s^-(y)\,\xi_t^-(z)\,\middle|\,\mathcal{F}_{R_T}\right]
\;\le\; \P\!\left(\exists(u,v)\in
\mathcal{N}_s(-\rho\log R_T)\times
\mathcal{N}_t(-\rho\log R_T):\,
Q(u,v)\ge R_T\,\big|\,\mathcal{F}_{R_T}\right).
\]
By Theorem~\ref{thm:earlybr} {(applied with $x=0$)}, for every
$\varepsilon\in(0,\tfrac{1}{2}-\sqrt{2}\rho)$ and all
sufficiently large~$T$,
\[
\E[Y_s^-Y_t^-]
\;\le\; \exp\!\big(2(\log\log T)^{\beta q_-}\big)\,
R_T^{-1/2+\sqrt{2}\rho+\varepsilon}.
\]
The close-time contribution is bounded by
$4R_T\exp\!\big(2(\log\log T)^{\beta q_-}\big)/T$.
Combining,
\begin{equation}\label{eq:left_var_bound}
\E\!\left[\big(\overline{F}_T^{-,K}\big)^2\right]
\;\le\; \frac{4R_T\exp\!\big(2(\log\log T)^{\beta q_-}\big)}{T}
  \;+\; \exp\!\big(2(\log\log T)^{\beta q_-}\big)\,
    R_T^{-1/2+\sqrt{2}\rho+\varepsilon}.
\end{equation}
Along $T_n=\exp(n^\eta)$, the factor
$\exp\!\big(2(\log\log T_n)^{\beta q_-}\big)
=\exp\!\big(2(\eta\log n)^{\beta q_-}\big)$
is dominated as $n\to\infty$ by
$R_{T_n}^{1/2-\sqrt{2}\rho-\varepsilon}
=\exp\!\big((\tfrac{1}{2}-\sqrt{2}\rho-\varepsilon)
(\eta\log n)^b\big)$,
since $\beta q_-<b$ by~\eqref{eq:qminus_choice}.
Hence the right-hand side of~\eqref{eq:left_var_bound}
is summable in~$n$, and Borel--Cantelli gives
$\overline{F}_{T_n}^{-,K}\to0$ a.s.\ as $n\to\infty$,
for each fixed~$K$.

Combining with~\eqref{eq:left_cond_mean}
and extending to the full time axis via
Lemma~\ref{subtoful},
\begin{equation}\label{eq:left_tail_done}
\lim_{K\to\infty}\limsup_{T\to\infty}
\frac{1}{T}\int_0^T|f_K^-(\hat{M}_t)|\,dt=0
\qquad\text{a.s.}
\end{equation}

\medskip
\noindent\textbf{Conclusion.}\;
For every $K>0$,
\[
\frac{1}{T}\int_0^T f(\hat{M}_t)\,dt
= \frac{1}{T}\int_0^T f^{(K)}(\hat{M}_t)\,dt
  + \frac{1}{T}\int_0^T f_K^+(\hat{M}_t)\,dt
  + \frac{1}{T}\int_0^T f_K^-(\hat{M}_t)\,dt.
\]
By~\eqref{eq:bounded_core}, the first term converges a.s.\
as $T\to\infty$.
By~\eqref{eq:right_tail_done}
and~\eqref{eq:left_tail_done}, the tail contributions
vanish.
Letting $K\to\infty$ and applying dominated convergence,
\[
\frac{1}{T}\int_0^T f(\hat{M}_t)\,dt
\;\xrightarrow[T\to\infty]{}\;
\E\!\left[
f\!\left(\frac{\mathcal{G}+\log C_\star+\log Z_\infty}{\sqrt{2}}\right)
\,\middle|\,Z_\infty\right]
\qquad\text{a.s.}
\]
\end{proof}

\section{Brownian estimates and toolbox}

This section collects several useful results, primarily from Bramson \cite{Bramson1983ConvergenceOS} and Arguin et al. \cite{Arguin2016}.
Define the probability measure $\P_{0,x}^{t,y}$,
under which $B_t$ is a Brownian bridge from $x$ at time $0$ to $y$ at time $t$.
For a function $l : [s_1,s_2]\to \mathbb{R}$, we denote $B_l$ (or $B_l[s_1,s_2]$ in case of ambiguity) the set of paths lying strictly above $l$. Similarly, $B^l$ the set of paths lying strictly below $l$.

The following is a well-known expression for the probability that a Brownian bridge stays non-negative.
\begin{lemma}\label{bmbar}For $t>0$, and $x,y\geq0$,
\begin{equation}
    \P_{0,x}^{t,y}(B_s \geq 0 , s \in [0,t])= 1- e^{-\frac{2xy}{t}}.
\end{equation}
\end{lemma}

The following result provides an upper bound on the probability that a Brownian bridge stays below an affine function.
\begin{lemma}[Lemma 3.4, \cite{Arguin2016}]\label{34Arguin}
Let \( Z_1, Z_2 \geq 0 \) and \( r_1, r_2 \geq 0 \). Then, for \( t > r_1 + r_2 \),
\[
\mathbb{P}_{0,0}^{t,0} \left( B_s \leq \left(1 - \frac{s}{t} \right) Z_1 + \frac{s}{t} Z_2, \quad \forall s \in [r_1, t - r_2] \right) \leq \frac{2}{t - r_1 - r_2} \prod_{i=1}^2 \left( Z(r_i) + \sqrt{r_i} \right),
\]
where
\[
Z(r_1) := \left(1 - \frac{r_1}{t} \right) Z_1 + \frac{r_1}{t} Z_2, \quad Z(r_2) := \frac{r_2}{t} Z_1 + \left(1 - \frac{r_2}{t} \right) Z_2.
\]
\end{lemma}

The following identity, derived from the reflection principle, gives the probability that a Brownian bridge remains non-negative over an interval.
\begin{lemma}\label{reflectionprinciple0}
For $y>0$ and $0<r<\gamma$,
\begin{equation}
    \P(\forall s \in [r, \gamma] \, , \, B_s \geq 0 \mid B_\gamma = y ) = 1- 2\P(B_r \leq 0 \mid B_\gamma = y ).
\end{equation}
\end{lemma}

We now state a useful monotonicity result.
\begin{lemma}[Lemma 2.6, \cite{Bramson1983ConvergenceOS}]\label{monoticitypath}
    Assume that $l_1(s)$, $l_2(s)$, and $\wedge(s)$ satisfy $l_1(s)\leq l_2(s) \leq \wedge(s)$ for $s\in [0,t]$, and that $\P(B_{l_2}[0,t])>0$. Then,
    \begin{equation}
        \P_{0,0}^{t,0}(B^{\wedge}\mid B_{l_1})\geq \P_{0,0}^{t,0}(B^{\wedge}\mid B_{l_2}),
    \end{equation}
    and
    \begin{equation}
        \P_{0,0}^{t,0}(B_{\wedge}\mid B_{l_1})\leq \P_{0,0}^{t,0}(B_{\wedge}\mid B_{l_2}).
    \end{equation}
\end{lemma}
Define the function $\Gamma_{t,\alpha}(s) \coloneqq (s \wedge (t-s))^\alpha$ for $\alpha \geq 0$. The following result establishes that if a Brownian bridge remains above \(-\Gamma_{t,\alpha}\) for some \(\alpha \in (0,1/2)\) on the interval \([r, t - r]\), then it remains above \(\Gamma_{t,\alpha}\) on the same interval.
\begin{lemma}[Proposition 6.1, \cite{Bramson1983ConvergenceOS}]\label{prop61bramson}
    Let $C>0$ and for all $t$, $L_t : [0,t]\to \mathbb{R}$ be a family of functions.
For any $\alpha \in (0,1/2)$, suppose there exists an $r_0$ such that $L_t(s) \leq C\Gamma_{t,\alpha}(s)$ for all $s \in [r_0,t-r_0]$ and for all $t$. Then for any $\delta>0$, there exists $r_1$ such that for all $r\geq r_1$  and $t\geq 3r$,
\begin{equation}\label{onfr5}
        \frac{\P_{0,0}^{t,0}( \forall s \in [r,t-r] : B_s > L_t(s) +C\Gamma_{t,\alpha}(s))}{\P_{0,0}^{t,0}( \forall s \in [r,t-r] : B_s > L_t(s) -C\Gamma_{t,\alpha}(s))} \geq 1- r^{\alpha-1/2 +\delta}.
\end{equation}
\end{lemma}

\begin{proof}
Inequality~\eqref{onfr5} follows directly from the proof of Proposition~6.1 in~\cite{Bramson1983ConvergenceOS}. For any $\epsilon \in (1/2,1-\alpha)$, there exist constants $C_3$, $a_1$, and $a_2$ such that for large $r$, the LHS of \eqref{onfr5} is greater than 
\begin{equation}\label{joig6}
    \left(e^{-C_3r^{\alpha +\epsilon-1}}\right)\left(1-a_1\sum_{k=\ceil{r}}^{+\infty}ke^{-a_2k^{2\epsilon -1}}\right),
\end{equation}
uniformly in $t$. This lower bound is derived from Eqs.~(6.9) and (6.10) and Lemmas 2.6 and 2.7 in~\cite{Bramson1983ConvergenceOS}. Therefore, for any given $\delta>0$, choosing $\epsilon \in (1/2, \min(1/2+\delta, 1-\alpha))$ in the lower bound \eqref{joig6} yields inequality~\eqref{onfr5} for sufficiently large $r$.
\end{proof}

The following result, which is proved via a coupling argument, shows that the ratio considered in Lemma \ref{prop61bramson} is increasing with respect to the endpoint.
\begin{lemma}[Proposition 6.3, \cite{Bramson1983ConvergenceOS}]\label{prop63bramson} Assume that the functions $l_1$ and $l_2$ are upper semi-continuous except at most finitely many points, and that $l_1(s)\leq l_2(s)$ for $s\in [0,t]$. Then, provided that the denominator is nonzero,
\begin{equation}
      \begin{aligned}
               & \frac{\P_{0,0}^{t,y}( \forall s \in [r,t-r] : B_s > l_2(s))}{\P_{0,0}^{t,y}( \forall s \in [r,t-r] : B_s > l_1(s))}     & & \text{ is increasing in $y$.}\\
      \end{aligned}
    \end{equation}
\end{lemma}

The following are standard results (see, e.g., \cite{HarrisRoberts2017}) in the theory of branching processes, which enable one to compute the first and second moments by reducing the problem to a single-particle or a two-particle system, respectively.

\begin{lemma}[Many-to-one Lemma]
    For any $t\geq 0$ and any measurable function $f: C([0,t])  \rightarrow \mathbb{R^+}$,
\begin{equation*}
    \E\left[\sum_{u\in \mathcal{N}_t}f\bigl((X_u(r))_{0\le r\le t}\bigr)\right]= e^t \E\left[ f\bigl((B_r)_{0\le r\le t}\bigr)\right],
\end{equation*}
where $(B_r)_{r\ge 0}$ is a standard Brownian motion.
\end{lemma}

\begin{lemma}[Many-to-two Lemma]\label{lem:many-to-two-two-times}
Let $f: C([0,t],\mathbb R)\to\mathbb R^+$ and $g: C([0,s],\mathbb R)\to\mathbb R^+$ be measurable, and assume $0\le s\le t$. Then
\begin{equation}
\begin{aligned}
&\E\!\left[ \Bigl(\sum_{u\in\mathcal N_t} f\bigl((X_u(r))_{0\le r\le t}\bigr)\Bigr)
               \Bigl(\sum_{v\in\mathcal N_s} g\bigl((X_v(r))_{0\le r\le s}\bigr)\Bigr) \right] \nonumber\\
&\qquad= e^{\,t}\, \E\!\left[ f\bigl((B_r)_{0\le r\le t}\bigr)\, g\bigl((B_r)_{0\le r\le s}\bigr) \right]\\
       &\qquad + \int_0^{s} 2\,e^{\,t+s-\gamma}\,
          \E\!\left[ f\bigl((B^{(1,\gamma)}_r)_{0\le r\le t}\bigr)\,
                     g\bigl((B^{(2,\gamma)}_r)_{0\le r\le s}\bigr) \right] \, \mathrm{d}\gamma,
\label{eq:two-time-many-to-two}
\end{aligned}
\end{equation}
where for each $\gamma\in[0,s]$ the pair 
$\bigl((B^{(1,\gamma)}_r)_{r\ge 0},(B^{(2,\gamma)}_r)_{r\ge 0}\bigr)$ are Brownian motions that \emph{coincide up to time $\gamma$} and evolve independently thereafter.
\end{lemma}

\section{Proof of intermediate lemmas}\label{Sec:proofprop}

\begin{proof}[Proof of Lemma~\ref{lem:BK-cross}]
Set \(A:=\bigcup_{u\in U}A_u\) and \(E:=\bigcup_{u\in U}E_u\).
Fix \(\omega\in \bigcup_{u\neq v}(A_u\cap E_v)\). Then there exist \(u_0\neq v_0\) with
\(\omega\in A_{u_0}\cap E_{v_0}\).

By the independence assumption, each \(A_u\) depends only on the variables associated with index \(u\),
and likewise each \(E_u\) depends only on those of index \(u\).
Define the witness index sets \(I:=\{u_0\}\) for \(A\) and \(J:=\{v_0\}\) for \(E\).
Since \(\omega\in A_{u_0}\), fixing the coordinates of index \(u_0\) as in \(\omega\) forces
membership in \(A_{u_0}\subseteq A\); similarly, \(\omega\in E_{v_0}\) yields a witness for \(E\).
Because \(u_0\neq v_0\), the sets \(I\) and \(J\) are disjoint, so
\[
\bigcup_{u\neq v}(A_u\cap E_v)\;\subseteq\; A\circ E,
\]
where \(\circ\) denotes disjoint occurrence.

By the van den Berg--Kesten--Reimer inequality,
\(\mathbb{P}(A\circ E)\le \mathbb{P}(A)\,\mathbb{P}(E)\).
Combining with the inclusion gives
\[
\mathbb{P}\!\left(\bigcup_{u\neq v} A_u\cap E_v\right)
\le \mathbb{P}(A\circ E)
\le \mathbb{P}\!\left(\bigcup_{u} A_u\right)\,
   \mathbb{P}\!\left(\bigcup_{u} E_u\right),
\]
as claimed.
\end{proof}

\begin{proof}[Proof of Lemma \ref{subtoful}]
Fix $n$ and any $T\in[S_n,S_{n+1}]$. Let $y(t) = x(t) - M \ge 0$, and define its time average $\tilde{\rho}_T := \rho_T - M$. Since $y(t) \ge 0$, the integral $\int_0^T y(t)\,dt$ is non-decreasing, yielding
\[
\int_0^{S_n} y(t)\,dt \le \int_0^T y(t)\,dt \le \int_0^{S_{n+1}} y(t)\,dt.
\]
Dividing by $T$ and using $S_n \le T \le S_{n+1}$, we obtain
\[
\frac{S_n}{S_{n+1}}\,\tilde{\rho}_{S_n} \le \tilde{\rho}_T \le \frac{S_{n+1}}{S_n}\,\tilde{\rho}_{S_{n+1}}.
\]
As $n\to\infty$, we have $\tilde{\rho}_{S_n} \to \rho - M$ and $S_{n+1}/S_n \to 1$, so both the lower and upper bounds tend to $\rho - M$. Hence $\tilde{\rho}_T \to \rho - M$, which implies $\rho_T \to \rho$ as $T\to\infty$.
\end{proof}

Before proving Lemma \ref{gih}, we first recall the precise bound obtained in Theorem 2.2 of \cite{Arguin2016}.

\begin{lemma}[Theorem~2.2 in~\cite{Arguin2016}]\label{lemmarg}
Let \(0<\gamma<\tfrac12\) and \(y\in\mathbb{R}\). There exist constants \(C>0\) and \(r_0=r_0(\gamma,y)>0\) such that, for all \(r\ge r_0\) and \(t\ge 3r\),
\begin{equation}\label{jth}
\begin{split}
\mathbb{P}\!\Bigl(\exists\, s\in [r,t-r],\ \exists\, u\in \mathcal{N}_t:\ 
& X_u(s)>\tfrac{s}{t}\,m_t+\Gamma_{t,\gamma}(s)+y\Bigr) \\
&\le C\, r^{\gamma/2}\, e^{-\sqrt{2}\, r^{\gamma/2}}.
\end{split}
\end{equation}
\end{lemma}

\begin{proof}
Inequality~\eqref{jth} follows directly from the proof of Theorem~2.2 in~\cite{Arguin2016}: 
summing the estimates provided by Eqs.~(5.8) and~(5.18) there yields the stated right-hand side.
\end{proof}
Denote by \( p_t(x) = (2\pi t)^{-1/2} e^{-x^2/(2t)} \) the Gaussian density with mean \(0\) and variance \(t\).
We now turn to the proof of Lemma~\ref{gih}.
\begin{proof}[Proof of Lemma \ref{gih}] 
Let $0<\alpha<1/2$ and $\delta>0$. For $R\geq r\geq 0$, denote $I_r=[r,t-r]$, $I_R=[R,t-R]$. Choose $\gamma\in(0,\alpha)$. By Lemma \ref{lemmarg} (with $y=0$), there exists $C_0$ and $r_0$ such that for $r \geq r_0$ and $t \ge 3r$,
\begin{equation}\label{do1}
    \P\left(\exists\, s\in [r,t-r],\ \exists\, u\in \mathcal{N}_t:\ X_u(s)>\tfrac{s}{t}\,m_t+\Gamma_{t,\gamma}(s)\right)< C_0\, r^{\gamma/2}\, e^{-\sqrt{2}\, r^{\gamma/2}}.
\end{equation}
Moreover, Bramson, in Proposition~3 of his seminal work~\cite{Bramsondep}, derived the following upper bound on the right tail of the maximal displacement. There exists a constant $C_1>0$ such that, for all $t\ge0$ and $0\le y\le \sqrt{t}$,
\[
  \mathbb{P}\!\left(M_t \ge m_t + y\right)
  \;\le\;
  C_1\, (y+1)^2\, e^{-\sqrt{2}\,y}.
\]
Therefore, for all $r \geq 1$ and $t \geq 3r$,
\begin{equation}\label{gr5f}
  \mathbb{P}\!\left(\exists  u\in \mathcal{N}_t : X_u(t) \ge m_t + r^{\gamma/2}\right)
  \;\le\;
  C_1\, (r^{\gamma/2}+1)^2\, e^{-\sqrt{2}\,r^{\gamma/2}}.
\end{equation}

We are going to prove that, for any $\lambda>0$ and all sufficiently large $r,R$ and $t\ge 3r$,
\begin{equation}
\begin{split}\label{do2}
    \P\Bigl(&\exists u \in \mathcal{N}_t: X_u({t})-m_t \in \left[{x-\rho \log R},r^{\gamma/2}\right],\,\forall s \in I_r:\
    X_u(s) \leq \tfrac{s}{t}m_t+\Gamma_{t,\gamma}(s), \\
    & \exists s \in I_R: X_u(s) > \tfrac{s}{t}m_t-\Gamma_{t,\alpha}(s)\Bigr)
    \leq C {(\sqrt{r}+\rho \log R)^2} R^{\alpha +\sqrt{2}\rho -1/2 +\lambda},
\end{split}
\end{equation}
where $C$ is a constant {depending on $x$ and $\rho$}. Choosing $r = R^\epsilon$ for some {$\epsilon\in(0,\delta/2)$ and $\lambda\in (0,\delta/2)$}, it follows that for sufficiently large $R$,
\begin{equation}
    \text{\eqref{do1}} + \text{\eqref{gr5f}} + \text{\eqref{do2}} \;\leq\; R^{\alpha -1/2 +\sqrt{2}\rho +\delta},
\end{equation}
uniformly in $t$. This proves inequality~\eqref{ojer}.

It remains to establish the bound \eqref{do2}. To this end, we apply Markov’s inequality and the many-to-one Lemma, which yield that the probability in \eqref{do2} is bounded by
\begin{equation}\label{gfdgd}
\begin{split}
         e^t&\P\Bigl(B_t-m_t \in \left[{x-\rho \log R},r^{\gamma/2}\right],\,\forall s \in I_r:\
    B_s \leq \tfrac{s}{t}m_t+\Gamma_{t,\gamma}(s), \\
    & \qquad \exists s \in I_R: B_s > \tfrac{s}{t}m_t-\Gamma_{t,\alpha}(s)  \Bigr)\\& =e^t  \int_{-r^{\gamma/2}}^{{\rho \log R-x}} p_t(m_t -z)
           \P_{0,0}^{t,z}(\exists s \in I_R  :  B_s <  \Gamma_{t,\alpha}  \text{ and } \forall s\in I_r: B_s \geq - \Gamma_{t,\gamma})\mathop{}\!\mathrm{d}z.
\end{split}
\end{equation}

 Fix $z\in[-r^{\gamma/2},{\rho \log R-x}]$. The probability inside the integral in \eqref{gfdgd} can be factorised as
    \begin{align}\label{2termadbound}
       &\P_{0,0}^{t,z}(\exists s \in I_R \, , \, B_s <  \Gamma_{t,\alpha}   \cap \forall s\in I_r, B_s \geq - \Gamma_{t,\gamma})\\
        &\quad =\P_{0,0}^{t,z}( \forall s\in I_r, B_s \geq - \Gamma_{t,\gamma})\P_{0,0}^{t,z}\left(\exists s \in I_R \, , \, B_s<\Gamma_{t,\alpha} \mid \forall s\in I_r, B_s \geq - \Gamma_{t,\gamma}\right).\notag 
    \end{align}
We deal with each of the two factors in turn. By monotonicity and Lemma \ref{34Arguin}, for $r$ large enough and all $t\geq 3r$,
    \begin{equation}
        \P_{0,0}^{t,z} \left( \forall s \in I_r,\ B_s \geq 0 \right)\leq 2 \frac{\left( \sqrt{r} + {(\rho \log R-x)^+} \right)\left( \sqrt{r} + r \tfrac{{(\rho \log R-x)^+}}{t} \right)}{t - 2r}\leq \frac{{C(\sqrt{r}+\rho \log R)^2}}{t}.
    \end{equation}
Moreover, by Lemmas \ref{prop63bramson} and \ref{prop61bramson}, for $r$ large enough and all $t\geq 3r$,
    \begin{equation}
        \begin{split}
            \P_{0,0}^{t,z}& \left( \forall s \in I_r,\ B_s \geq 0 \mid \forall s \in I_r,\ B_s \geq -\Gamma_{t,\gamma} \right)^{-1}\\
            &\leq \P_{0,0}^{t,-r^{\gamma/2}} \left( \forall s \in I_r,\ B_s \geq 0 \mid \forall s \in I_r,\ B_s \geq -\Gamma_{t,\gamma} \right)^{-1}\leq 2.
        \end{split}
    \end{equation}
  
Therefore, 
    \begin{equation}
    \begin{split}
        \P_{0,0}^{t,z} &\left( \forall s \in I_r,\ B_s \geq -\Gamma_{t,\gamma} \right) 
        \\ &= \P_{0,0}^{t,z} \left( \forall s \in I_r,\ B_s \geq 0 \right)
         \P_{0,0}^{t,z} \left( \forall s \in I_r,\ B_s \geq 0 \mid \forall s \in I_r,\ B_s \geq -\Gamma_{t,\gamma} \right)^{-1}
        \\& \leq \frac{{C(\sqrt{r}+\rho \log R)^2}}{t} \quad \text{for } r \text{ large enough and all $t\geq 3r$.}
    \end{split}   
\end{equation}
We now bound the second term in \eqref{2termadbound}. For any $\lambda>0$,
    \begin{equation}
      \begin{aligned}
        \P_{0,0}^{t,z}&(\exists s \in I_R : \, B_s < \Gamma_{t,\alpha}  \mid \forall s\in I_r: B_s \geq - \Gamma_{t,\gamma}) & &  \\
               & \leq \P_{0,0}^{t,z}(\exists s \in I_R : \, B_s <  \Gamma_{t,\alpha} \mid \forall s\in I_R: B_s \geq - \Gamma_{t,\alpha}) & & \text{\footnotesize(Lemma \ref{monoticitypath})}\\
                & = 1-\P_{0,0}^{t,z}(\forall s \in I_R: \, B_s \geq  \Gamma_{t,\alpha} \mid \forall s\in I_R: B_s \geq - \Gamma_{t,\alpha}) & & \\
                 & \leq 1-\P_{0,0}^{t,-r^{\gamma/2}}(\forall s \in I_R : \, B_s \geq  \Gamma_{t,\alpha} \mid \forall s\in I_R: B_s \geq - \Gamma_{t,\alpha}) & & \text{\footnotesize(Lemma \ref{prop63bramson})}\\
                 & \leq R^{\alpha-1/2 +\lambda} \qquad \text{for $R$ large enough.} & & \text{\footnotesize(Lemma \ref{prop61bramson})}
      \end{aligned}
    \end{equation}
    
Therefore, for any $\lambda>0$, for $r$ and $R$ large enough, \eqref{2termadbound} is bounded by
    \begin{equation}
        \frac{{C(\sqrt{r}+\rho \log R)^2}}{t}R^{\alpha-1/2 +\lambda}.
    \end{equation}

Thus, for $r$ and $R$ large enough and $t\geq 3r$, equation \eqref{gfdgd} is smaller than
    \begin{equation}\label{48}
        \begin{split}
           {C(\sqrt{r}+\rho \log R)^2} R^{\alpha-1/2 +\lambda}\int_{-r^{\gamma/2}}^{{\rho \log R-x}} e^{\frac{m_t}{t} z -\frac{z^2}{2t}} \mathop{}\!\mathrm{d}z\leq {C(\sqrt{r}+\rho \log R)^2}R^{\alpha-1/2 +\lambda +\sqrt{2}\rho},
    \end{split}
    \end{equation}
for some constant $C$. This completes the proof of \eqref{do2}.
\end{proof}

The following result will directly imply Lemma \ref{lemma-earlybrloc} by Markov’s inequality.

\begin{lemma}
Let $0<\alpha<1/2${, $x\in\mathbb{R}$,} and {$\rho\geq 0$}. There exist constants $C=C({x,}\rho,\alpha)$ and $r_0$ such that, for all $r\ge r_0$, $R\geq2r$, $s\ge3r$ and $t\ge s+2r$,
\begin{equation}\label{eq:path-localization-br-E}
\begin{split}
\E\!\Bigl[&\#\Bigl\{(u,v)\in\mathcal N_s({x-\rho \log r})\times\mathcal N_t({x-\rho \log r})\colon
u \text{ is $\alpha$--$s$--localised on }[r,s-r], \\
&\quad v \text{ is $\alpha$--$t$--localised on }[r,t-r],
\text{ and } Q(u,v)\ge R\Bigr\}\Bigr] \\
&
\;\le\; Cr^{2\sqrt2\rho+\frac12}\left(r e^{-\sqrt{2}R^\alpha}+(t-s)e^{-\sqrt{2}(t-s)^\alpha}\right).
\end{split}
\end{equation}
\end{lemma}
\begin{proof}
Throughout the proof, we fix \(0<\alpha<1/2\){, \(x\in\mathbb{R}\),} and \({\rho\ge0}\), and let \(r, R, s, t\) be parameters satisfying \(r>0\), \(R\ge2r\), \(s\ge3r\), and \(t\ge s+2r\); we will eventually let \(r\to\infty\).
By Lemma \ref{lem:many-to-two-two-times}, the expectation in \eqref{eq:path-localization-br-E} equals
\begin{equation}\label{fdfo}
\begin{aligned}
&2\int_{R}^{s} e^{s+t-\gamma}\;
\P\!\left(
\begin{aligned}
&B^{1,\gamma}_s \ge m_s{+x}-\rho \log r,\ \forall \phi\in [r,s-r],\ B^{1,\gamma}_\phi \le \tfrac{m_s}{s}\phi-\Gamma_{s,\alpha}(\phi),\\
&B^{2,\gamma}_t \ge m_t{+x}-\rho \log r,\ \forall \phi\in [r,t-r],\ B^{2,\gamma}_\phi \le \tfrac{m_t}{t}\phi-\Gamma_{t,\alpha}(\phi)
\end{aligned}
\right)\,d\gamma \\
&\qquad +\ e^{t}\;\P\!\left(
\begin{aligned}
&B_s\ge m_s{+x}-\rho \log r,\ \forall \phi\in[r,s-r],\ B_\phi \le \tfrac{m_s}{s}\phi-\Gamma_{s,\alpha}(\phi),\\
&B_t\ge m_t{+x}-\rho \log r,\ \forall \phi\in[r,t-r],\ B_\phi \le \tfrac{m_t}{t}\phi-\Gamma_{t,\alpha}(\phi)
\end{aligned}
\right),
\end{aligned}
\end{equation}
where for $\gamma \in \left[0,t\right]$, $(B^{(1,\gamma)}_\phi)_{\phi\in \mathbb{R}}$ and $(B^{(2,\gamma)}_\phi)_{\phi\in \mathbb{R}}$ are coupled Brownian motions sharing a common path until time $\gamma$, then evolving independently thereafter.
We first show that, in our parameter regime, the second term in \eqref{fdfo} vanishes when \(r\) is sufficiently large. 
We then bound the contribution of the integral term in \eqref{fdfo}.

\medskip
\noindent\textbf{Bound for the common ancestor term.}
The second contribution in \eqref{fdfo} equals
\begin{equation}\label{eq:anc-def}
e^{t}\,\P\!\left(
\begin{aligned}
&B_s\ge m_s{+x}-\rho \log r,\ \forall \phi\in[r,s-r]:\ B_\phi \le \tfrac{m_s}{s}\phi-\Gamma_{s,\alpha}(\phi),\\
&B_t\ge m_t{+x}-\rho \log r,\ \forall \phi\in[r,t-r]:\ B_\phi \le \tfrac{m_t}{t}\phi-\Gamma_{t,\alpha}(\phi)
\end{aligned}
\right).
\end{equation}
Note that the constraint at $\phi=s$ enforces
\begin{equation}\label{condit}
B_s \le \tfrac{m_t}{t}\,s - \Gamma_{t,\alpha}(s), \text{ and that }B_s\ge m_s{+x}-\rho \log r.
\end{equation}
Thus, if
\begin{equation}\label{eq:anc-vanish}
\Gamma_{t,\alpha}(s) > \tfrac{3}{2\sqrt2}\bigl(\log s-\tfrac{s}{t}\log t\bigr)+\rho \log r {-x},
\end{equation}
the conditions \eqref{condit} are incompatible, and therefore \eqref{eq:anc-def} vanishes. We now establish \eqref{eq:anc-vanish}. If $s\leq t/2$, \eqref{eq:anc-vanish} follows from $\Gamma_{t,\alpha}(s)=s^\alpha \gg \log(s) \vee \log(r)$ for $s$ large enough. 
Observe that for $r$ large enough, we have
\begin{equation}\label{eq_log_ts}
    \log t \leq \tfrac{s}{t}\log t + \log(t - s).
\end{equation}
Therefore, for $r$ large enough, when $s\ge t/2$,
\begin{equation*}
    \Gamma_{t,\alpha}(s)= (t-s)^\alpha\gg 10\log(t-s)\gg \tfrac{3}{2\sqrt{2}}\left( \log s -\tfrac{s}{t}\log t\right) +\rho \log r {-x}.
\end{equation*}
Thus, for $r$ large enough and any $s\geq 3r$, and $t\geq s+2r$, \eqref{eq:anc-vanish} is established.

\medskip
\noindent\textbf{Bound for the split term.}
For the integral contribution in \eqref{fdfo}, conditioning on the common position at time $\gamma$ and using the Markov property yields the bound
\begin{equation}\label{eq:split-term-pre}
\begin{aligned}
&2\int_{R}^{s} e^{s+t-\gamma}
\int_{-\infty}^{\tfrac{m_t}{t}\gamma-\Gamma_{t,\alpha}(\gamma)}
p_\gamma(y)\;
\P\!\left(\forall \phi\in[r,\gamma],\;
B_\phi \le \tfrac{m_t}{t}\phi-\Gamma_{t,\alpha}(\phi)\,\bigm|\,B_\gamma=y\right)\\
&\qquad\times
\P\!\left(B_{s-\gamma}\ge m_s{+x}-\rho \log r,\;
\forall \phi\in[\gamma,s-r],\;
B_{\phi-\gamma}\le \tfrac{m_s}{s}\phi-\Gamma_{s,\alpha}(\phi)\,\bigm|\,B_0=y\right)\\
&\qquad\times
\P\!\left(B_{t-\gamma}\ge m_t{+x}-\rho \log r,\;
\forall \phi\in[\gamma,t-r],\;
B_{\phi-\gamma}\le \tfrac{m_t}{t}\phi-\Gamma_{t,\alpha}(\phi)\,\bigm|\,B_0=y\right)
\,\mathrm dy\,\mathrm d\gamma,
\end{aligned}
\end{equation}
After a standard re-centring of the path, this can be rewritten in the equivalent form
\begin{equation}\label{eq:split-term}
\begin{aligned}
&2\int_{R}^{s} e^{s+t-\gamma}
\int_{\tfrac{3}{2\sqrt{2}}\tfrac{\log t}{t}\gamma + \Gamma_{t,\alpha}(\gamma)}^{\infty}
p_\gamma(\sqrt2\gamma-y)\\
&\qquad\times
\P\!\left(\forall \phi\in[r,\gamma],\;
B_\phi \ge \tfrac{3}{2\sqrt{2}}\tfrac{\log t}{t}\phi + \Gamma_{t,\alpha}(\phi)\,\bigm|\,B_\gamma=y\right)\\
&\qquad\times
\P_{\sqrt2\gamma-y}\!\left(B_{s-\gamma}\ge m_s{+x}-\rho \log r,\;
\forall \phi\in[\gamma,s-r],\;
B_{\phi-\gamma}\le \tfrac{m_s}{s}\phi-\Gamma_{s,\alpha}(\phi)\right)\\
&\qquad\times
\P_{\sqrt2\gamma-y}\!\left(B_{t-\gamma}\ge m_t{+x}-\rho \log r,\;
\forall \phi\in[\gamma,t-r],\;
B_{\phi-\gamma}\le \tfrac{m_t}{t}\phi-\Gamma_{t,\alpha}(\phi)\right)
\,\mathrm dy\,\mathrm d\gamma.
\end{aligned}
\end{equation}

We bound each term of the last integral in equation \eqref{eq:split-term} separately. Let $\gamma \in [R,s]$ and $y\geq \tfrac{3}{2\sqrt2}\tfrac{\log t}{t}\gamma+\Gamma_{t,\alpha}(\gamma)
$. For the first term, we have that,

\begin{equation}
    p_\gamma\left(\sqrt{2}\gamma-y\right)= \frac{1}{\sqrt{2 \pi \gamma}}e^{-\gamma +\sqrt{2}y- \frac{y^2}{2\gamma}} \leq \frac{1}{\sqrt{2 \pi \gamma}}e^{-\gamma +\sqrt{2}y}.
\end{equation}

For the second term, let
\[
a(\phi)=\frac{3}{2\sqrt{2}}\frac{\log t}{t}\,\phi+\Gamma_{t,\alpha}(\phi),\qquad 0\le \phi\le\gamma,
\]
which is concave with \(a(0)=0\). Hence, for \(0<r<\gamma\),
\[
\{\forall \phi\in[r,\gamma]: B_\phi\ge a(\phi)\}
\subseteq
\Big\{\forall \phi\in[r,\gamma]: B_\phi\ge \tfrac{\phi}{\gamma}a(\gamma)\Big\}.
\]
Define the detrended bridge \(\widetilde B_\phi:=B_\phi-\tfrac{\phi}{\gamma}a(\gamma)\).
Conditioning on \(B_\gamma=y\), \(\widetilde B\) is a Brownian bridge from \(0\) to \(y-a(\gamma)\) on \([0,\gamma]\).
By Lemma~\ref{reflectionprinciple0},
\[
\P\!\big(\forall \phi\in[r,\gamma]: B_\phi\ge0\big)
=1-2\,\P(\widetilde B_r<0)
\le \sqrt{\tfrac{2}{\pi}}\;\frac{\mu_+}{\sigma},
\]
with \(\mu=\tfrac{r}{\gamma}\big(y-a(\gamma)\big)\) and \(\sigma^2=\tfrac{r(\gamma-r)}{\gamma}\).
Therefore,
\[
\P\!\left(\forall \phi\in[r,\gamma]\ B_\phi\ge a(\phi)\ \Big|\ B_\gamma=y\right)
\le \sqrt{\tfrac{2}{\pi}}\;\frac{\sqrt r}{\sqrt{\gamma(\gamma-r)}}\,
\Big(y-\tfrac{3}{2\sqrt{2}}\tfrac{\log t}{t}\gamma-\Gamma_{t,\alpha}(\gamma)\Big)_+.
\]

For the third term, by monotonicity,
\begin{align}
   &\P_{\sqrt{2}\gamma-y}\!\left(
      B_{s-\gamma} \ge m_s {+x}-\rho \log r, \;
      \forall \phi \in [\gamma,s-r], \;
      B_{\phi-\gamma} \leq \tfrac{m_s}{s}\phi-\Gamma_{s,\alpha}(\phi)
   \right) \nonumber \\
   &\quad\leq \P_{\sqrt{2}\gamma-y}\!\left( B_{s-\gamma} \ge m_s {+x}-\rho \log r\right)\,
   \P_{\gamma,\sqrt{2}\gamma-y}^{s,m_s{+x}-\rho \log r}\!\left(
      \forall \phi \in [\gamma,s-r], \;
      B_{\phi} \leq \tfrac{m_s}{s}\phi-\Gamma_{s,\alpha}(\phi)
   \right).
\end{align}
Observe that for $\gamma\in [r,s]$, the function 
\begin{equation}
    \gamma \longmapsto  \Gamma_{t,\alpha}(\gamma) -\gamma\tfrac{3}{2\sqrt{2}}\left(\tfrac{\log s}{s}-\tfrac{\log t}{t}\right) +x -\rho \log r 
\end{equation}
is concave in $\gamma$, therefore its minimum on $[r,s]$ is attained at $r$ or $s$. On one hand we established \eqref{eq:anc-vanish}. On the other hand, since $r^\alpha \gg \log r$  and $\tfrac{\log s}{s}\le \tfrac{\log r}{r}$, we obtain
\begin{equation}
    \Gamma_{t,\alpha}(r) -r\tfrac{3}{2\sqrt{2}}\left(\tfrac{\log s}{s}-\tfrac{\log t}{t}\right) +x -\rho \log r \gg 1.
\end{equation}
Therefore, for $y\ge \tfrac{3}{2\sqrt{2}}\tfrac{\log t}{t}\gamma + \Gamma_{t,\alpha}(\gamma)$ and $r$ sufficiently large,
\begin{equation}
    \sqrt{2}(s-\gamma) - \tfrac{3}{2\sqrt{2}}\log s + y +x-\rho \log r\ge \sqrt{2}(s-\gamma)- \tfrac{3}{2\sqrt{2}} (s-\gamma)\tfrac{\log s}{s}\ge s-\gamma.
\end{equation}
As a consequence,
\begin{align}
   \P_{\sqrt{2}\gamma-y}\!\left( B_{s-\gamma} \ge m_s {+x}-\rho \log r\right)
   &= \frac{1}{\sqrt{2\pi(s-\gamma)}}
      \int_{{x}-\rho \log r}^{\infty}
      \exp\!\left\{
        -\tfrac{\big(\sqrt{2}(s-\gamma) - \tfrac{3}{2\sqrt{2}}\log s + y + z\big)^2}{2(s-\gamma)}
      \right\}\,\mathrm{d}z \nonumber \\
   &\leq \frac{1}{\sqrt{s-\gamma}}
      \,r^{\sqrt{2}\rho}\,{e^{-\sqrt{2}x}}\,e^{\gamma-s-\sqrt{2}y}\,s^{3/2}.
\end{align}

Moreover,
\begin{equation}\label{eq:bridge-concave} 
\begin{aligned}
   \P&_{\gamma,\sqrt{2}\gamma-y}^{s,m_s{+x}-\rho \log r}\!\left( 
      \forall \phi \in [\gamma,s-r],\;
      B_{\phi} \le \tfrac{m_s}{s}\phi-\Gamma_{s,\alpha}(\phi)
   \right)\\ 
   &= \P_{\gamma,\tfrac{3}{2\sqrt{2}}\tfrac{\log s}{s}\gamma + \Gamma_{s,\alpha}(\gamma)-y}^{s,{x}-\rho \log r}\!\left(
      \forall \phi \in [\gamma,s-r],\;
      B_{\phi} \le \Gamma_{s,\alpha}(\gamma)\tfrac{s-\phi}{s-\gamma}-\Gamma_{s,\alpha}(\phi)
   \right) \\
   &\le \P_{\gamma,\tfrac{3}{2\sqrt{2}}\tfrac{\log s}{s}\gamma + \Gamma_{s,\alpha}(\gamma)-y}^{s,{x}-\rho \log r}\!\left(
      \forall \phi \in [\gamma,s-r],\;
      B_{\phi} \le 0
   \right).
\end{aligned}
\end{equation}

The last inequality follows from the concavity of $\Gamma_{s,\alpha}$ and that $\Gamma_{s,\alpha}(s)=0$.
Applying Lemma~\ref{34Arguin} when $s-\gamma\geq2r$ and $r>(\rho \log r)^2+1$ then gives the upper bound
\begin{equation}\label{eq:concave-arguin-bound}
  \frac{12\,\sqrt{r}}{\,s-\gamma\,}\,
  \Big(y-\tfrac{3}{2\sqrt{2}}\tfrac{\log s}{s}\gamma - \Gamma_{s,\alpha}(\gamma)\Big)_+.
\end{equation}

Identically, for the fourth term we obtain
\begin{align}
   &\P_{\sqrt{2}\gamma-y}\!\left(
      B_{t-\gamma} \ge m_t {+x}- \rho \log r, \;
      \forall \phi \in [\gamma,t-r], \;
      B_{\phi-\gamma} \leq \tfrac{m_t}{t}\phi - \Gamma_{t,\alpha}(\phi)
   \right) \nonumber \\
   &\quad\leq \P_{\sqrt{2}\gamma-y}\!\left( B_{t-\gamma} \ge m_t {+x}-\rho \log r\right)\,
   \P_{\gamma,\sqrt{2}\gamma-y}^{\,t,m_t{+x}-\rho \log r}\!\left(
      \forall \phi \in [\gamma,t-r], \;
      B_\phi \leq \tfrac{m_t}{t}\phi - \Gamma_{t,\alpha}(\phi)
   \right).
\end{align}
Furthermore,
\[
\P_{\sqrt{2}\gamma-y}\!\left( B_{t-\gamma} \ge m_t {+x}-\rho \log r\right) 
\;\leq\; \frac{1}{\sqrt{t-\gamma}} \,
   r^{\sqrt{2}\rho}\, {e^{-\sqrt{2}x}}\,e^{\gamma-t-\sqrt{2}y}\, t^{3/2},
\]
and, using Lemma~\ref{34Arguin} when $t-\gamma\geq 2r$ and $r>(\rho \log r -x)^2+1$, we also have
\begin{equation}
    \P_{\gamma,\sqrt{2}\gamma-y}^{t,m_t{+x}-\rho \log r}\!\left(
      \forall \phi \in [\gamma,t-r], \;
      B_{\phi} \leq \tfrac{m_t}{t}\phi-\Gamma_{t,\alpha}(\phi)
   \right)\leq \frac{12\,\sqrt{r}}{\,t-\gamma\,}\,
  \Big(y-\tfrac{3}{2\sqrt{2}}\tfrac{\log t}{t}\gamma - \Gamma_{t,\alpha}(\gamma)\Big)_+.
\end{equation}

We partition the domain of integration of $\gamma$ in \eqref{eq:split-term-pre} into two subdomains and insert the estimates above. We first bound the contribution to \eqref{eq:split-term-pre} coming from the region $\gamma \le s-2r$. There exist constants $C$ and $r_0 = r_0({x,}\alpha, \rho)$ such that, for all $r \ge r_0$,
\begin{equation}\label{obu4}
    \begin{aligned}
&\int_{\gamma \le s-2r} (\cdots)\, d\gamma \le r^{2\sqrt{2}\rho}\int_R^{s-2r}  \left(\frac{rst}{(t-\gamma)\gamma(s-\gamma)}\right)^{3/2}\int_{\tfrac{3}{2\sqrt{2}}\tfrac{\log(t)}{t}\gamma+\Gamma_{t,\alpha}(\gamma)}^{+\infty}\\
    &\qquad \qquad\qquad\qquad \left(y-\tfrac{3}{2\sqrt{2}}\tfrac{\log t}{t}\gamma - \Gamma_{t,\alpha}(\gamma)\right)^2\left(y-\tfrac{3}{2\sqrt{2}}\tfrac{\log s}{s}\gamma \right)e^{-\sqrt{2}y}\mathop{}\!\mathrm{d}y \mathop{}\!\mathrm{d}\gamma\\
     &= Cr^{2\sqrt{2}\rho}\int_R^{s-2r}  \left(\frac{rste^{-\tfrac{\log t}{t}\gamma}}{(t-\gamma)\gamma(s-\gamma)}\right)^{3/2}\int_{\Gamma_{t,\alpha}(\gamma)}^{+\infty}\\
     &\qquad\qquad\qquad\left(y - \Gamma_{t,\alpha}(\gamma)\right)^2\left(y+\tfrac{3}{2\sqrt{2}}\left(\tfrac{\log t}{t}-\tfrac{\log s}{s}\right)\gamma\right)e^{-\sqrt{2}y}\mathop{}\!\mathrm{d}y \mathop{}\!\mathrm{d}\gamma.
    \end{aligned}
\end{equation}
Observe that for $r$ large enough, any $t\geq 3r$ and \( \gamma \in [r, t - r] \), we have
\begin{equation*}
    \log t \leq \tfrac{\gamma}{t}\log t + \log(t - \gamma).
\end{equation*}
Also, for $r$ large enough,
\begin{equation}
\begin{aligned}
    \int_{\Gamma_{t,\alpha}(\gamma)}^{+\infty}&\left(y - \Gamma_{t,\alpha}(\gamma)\right)^2\left(y+\tfrac{3}{2\sqrt{2}}\left(\tfrac{\log t}{t}-\tfrac{\log s}{s}\right)\gamma\right)e^{-\sqrt{2}y}\mathop{}\!\mathrm{d}y\\
    &\leq \int_{\Gamma_{t,\alpha}(\gamma)}^{+\infty}\left(y - \Gamma_{t,\alpha}(\gamma)\right)^2ye^{-\sqrt{2}y}\mathop{}\!\mathrm{d}y\leq \Gamma_{t,\alpha}(\gamma)e^{-\sqrt{2}\Gamma_{t,\alpha}(\gamma)}.
\end{aligned}
\end{equation}
Therefore, for $r$ large enough, the quantity \eqref{obu4} is smaller than
\begin{equation}\label{fg5}
\begin{aligned}
    Cr^{2\sqrt{2}\rho}\int_R^{s-2r}&  \left(\frac{rs}{\gamma(s-\gamma)}\right)^{3/2}\Gamma_{t,\alpha}(\gamma)e^{-\sqrt{2}\Gamma_{t,\alpha}(\gamma)}d\gamma\\
    &\leq Cr^{2\sqrt{2}\rho}\tfrac{2}{\alpha}\left(r^{3/2}\frac{e^{-\sqrt{2}R^\alpha}}{\sqrt{R}}+(t-s)e^{-\sqrt{2}(t-s)^\alpha}\right),
\end{aligned}
\end{equation}
by a standard Laplace method.

We now bound the contribution to \eqref{eq:split-term-pre} coming from the region $\gamma \ge s-2r$.
There exist constants $C>0$ and $r_0 = r_0({x,}\alpha, \rho)$ such that, for all $r \ge r_0$,

\begin{equation}\label{obu5}
    \begin{aligned}
&\int_{\gamma \ge s-2r} (\cdots)\, d\gamma \le Cr^{2\sqrt{2}\rho}\int_{s-2r}^s e^{s+t-\gamma} \int_{\tfrac{3}{2\sqrt{2}}\tfrac{\log(t)}{t}\gamma+\Gamma_{t,\alpha}(\gamma)}^{+\infty} p_\gamma\left(\sqrt{2}\gamma - y\right)\\ 
&\qquad \P\left(\forall \phi \in [r, \gamma], \, B_\phi \geq \tfrac{3}{2\sqrt{2}}\tfrac{\log(t)}{t}\phi +\Gamma_{t,\alpha}(\phi) \mid B_\gamma = y \right)\P_{\sqrt{2}\gamma-y}\left( B_{s-\gamma} \ge m_s {+x}-\rho \log r \right)\\
        &\qquad \times \P_{\sqrt{2}\gamma-y}\left( B_{t-\gamma} \ge m_t {+x}-\rho \log r, \, \forall \phi \in [\gamma,t-r], \, B_{\phi-\gamma} \leq \frac{m_t}{t}\phi-\Gamma_{t,\alpha}(\phi) \right)\mathop{}\!\mathrm{d}y \mathop{}\!\mathrm{d}\gamma\\
    &\leq C r^{2\sqrt{2}\rho+1}\int_{s-2r}^s \left(\frac{st}{(t-\gamma)\gamma}\right)^{3/2}\int_{\tfrac{3}{2\sqrt{2}}\tfrac{\log(t)}{t}\gamma+\Gamma_{t,\alpha}(\gamma)}^{+\infty}\frac{\Big(y-\tfrac{3}{2\sqrt{2}}\tfrac{\log t}{t}\gamma - \Gamma_{t,\alpha}(\gamma)\Big)^2}{(s-\gamma)^{1/2}}\\
    &\qquad \qquad\qquad\qquad\qquad\qquad\qquad e^{-\sqrt{2}y}\mathop{}\!\mathrm{d}y \mathop{}\!\mathrm{d}\gamma\\
     &\leq Cr^{2\sqrt{2}\rho +1}\int_{s-2r}^s \frac{s^{3/2}}{\gamma^{3/2}(s-\gamma)^{1/2}}\int_{0}^{+\infty}y^2e^{-\sqrt{2}y-\sqrt{2}\Gamma_{t,\alpha}(\gamma)}\mathop{}\!\mathrm{d}y \mathop{}\!\mathrm{d}\gamma\\
     &\leq Cr^{2\sqrt{2}\rho +1}\tfrac{3}{\sqrt{2}}\int_{s-2r}^s \frac{e^{-\sqrt{2}\Gamma_{t,\alpha}(\gamma)}}{(s-\gamma)^{1/2}} \mathop{}\!\mathrm{d}\gamma \leq  6Cr^{2\sqrt{2}\rho+3/2}e^{-\sqrt{2}\min(t-s,R)^\alpha}.
    \end{aligned}
\end{equation}
Summing the bounds obtained in \eqref{fg5} and \eqref{obu5} yields the RHS of \eqref{eq:path-localization-br-E}.

\end{proof}

\section{Discussion and related works}\label{sec:discussion}

\medskip

\noindent\textbf{Genealogy of the front of BBM.} 
Arguin, Bovier, and Kistler \cite{Arguin2016} showed that two extremal particles at time $t$ either branch very early or very recently. 
Later, \cite{flath1} established that particles chosen randomly within a sublinear distance of the front at time $t$ necessarily branched early. 
Theorem \ref{thm:earlybr} extends this picture to different times: if $t-s$ is large, extremal particles at times $t$ and $s$ must have branched early. 
This is consistent with \cite{Arguin2016}, since for bounded $t-s$, continuity still permits recent branching.

\medskip
\noindent\textbf{Ergodic convergence of functionals of extremal particles.}
Theorem~\ref{thm:functional_ergodic} covers scalar functionals of the recentred maximum. The same negative correlation principle extends to multi-particle functionals of the front, such as the Laplace functional studied in~\cite{arguibextremalergo}:
\[
\mathcal L_{t}(f)
:= \exp\!\Big(-\!\int_{\mathbb R} f(y)\,\mathcal E_{t}(dy)\Big),\quad
\mathcal E_{t}
:= \sum_{u\in\mathcal N_t} \delta_{\,X_u(t)-m_t},
\quad f\in C_c^+(\mathbb R).
\]
Let
\[
E:=\left\{\forall (u,v)\in\mathcal N_s(x)\times \mathcal N_t(x):\ Q(u,v)\le R\right\}.
\]
A negative-correlation result analogous to Proposition~\ref{thm:decorrelation} holds:
\[
        \P\!\left(\mathcal L_{s}(f)\ge x,\, \mathcal L_{t}(f)\ge y,\, E \,\middle|\, \mathcal{F}_{R}\right)
        \le
        \P\!\left(\mathcal L_{s}(f)\ge x \,\middle|\, \mathcal{F}_{R}\right)
        \P\!\left(\mathcal L_{t}(f)\ge y \,\middle|\, \mathcal{F}_{R}\right),
\]
so that the proof of Theorem~\ref{thm:fluctuations} carries over verbatim for $\mathcal L_{t}(f)$. The identification of the conditional limit was established in Proposition~9 of~\cite{arguibextremalergo}.

More generally, Theorem~\ref{thm:fluctuations} carries over to any reasonable functional of the front for which the analogue of~\eqref{eq:deco} holds. To complete an ergodic theorem for such a functional, it remains to identify the conditional limit; for this, it is useful to observe that, by combining Lemmas~\ref{gih} and~\ref{subtoful}, one may restrict one's attention to localised particles.
The negative correlation inequality (Proposition~\ref{thm:decorrelation}) also applies to $d$-dimensional branching Brownian motion, with the Euclidean norm replacing the absolute value. An ergodic theorem in that setting would further require generalizing the early branching estimate and studying the local fluctuations in higher dimension.


\medskip
\noindent\textbf{Almost sure ergodic vs almost sure pathwise localization.} Theorem \ref{thm:ergodic} was first established in \cite{arguin2012ergodic}. A key step in their proof is Theorem 3 in \cite{arguin2012ergodic} (Theorem \ref{thm:fluctuations} here), supported by several technical lemmas including Lemma 7 in \cite{arguin2012ergodic}. The proof of Lemma 7 in \cite{arguin2012ergodic} relies on Proposition 6 in \cite{arguin2012ergodic}, which asserts the following: For $0<\alpha<1/2$, there exist $r_0>0$ and $\delta>0$ such that, for all $r\ge r_0$ and $t\ge 3r$,
\[
\mathbb{P}\!\Bigl( \exists\, u \in \mathcal{N}_t(\mathcal{D}) \text{ not $\alpha$-$t$-localised on } [\,r,\,t-r\,]\Bigr) 
\;\le\; e^{-r^\delta}.
\]
This estimate is strong enough for a Borel–Cantelli argument, which implies that for $r_t=(20\log t)^{1/\delta}$,
\[
\bbone\Bigl\{
    \exists\, u \in \mathcal{N}_t(\mathcal{D}) 
    \text{ not $\alpha$-$t$-localised on } [\,r_t,\,t-r_t\,]
\Bigr\}
\;\xrightarrow[t\to\infty]{\text{a.s.}}\; 0 .
\]
However, as suggested by the mechanism in the proof of 
Theorem~1.3 in \cite{flath1}, such a pathwise localisation does not hold almost surely (although it 
does in probability). Consequently, Proposition~6 cannot be correct as stated. Importantly, Lemma~7 
itself is formulated in terms of an \emph{ergodic} localisation rather than a pathwise localisation. 
The distinction is crucial: while almost sure pathwise localisation fails, almost sure ergodic
localisation remains valid. This can be proved combining Lemma~\ref{gih} and Lemma~\ref{subtoful} as in the proof of Theorem~\ref{thm:fluctuations}.


The issue in Proposition~6 originates from a bound in \cite{Arguin2016}. In equation~(5.55) there, the multiplicative factor \(e^{\kappa_3 r^{\alpha+\theta-1}}\) appearing in equation~(5.49) is omitted\, which is immaterial for their purposes. The precise bound is derived in Lemma~\ref{gih} by introducing two time scales \(r\) and \(R\) implementing the entropic repulsion argument from \cite{Arguin2016}. The bound in Lemma~\ref{gih} is  sufficient for Theorem 2.3 in \cite{Arguin2016} to hold. However, it is not strong enough to yield an almost sure path localization using a Borel-Cantelli type argument.

\section*{Acknowledgements}

I am grateful to Julien Berestycki for his advice, insightful suggestions, and for pointing out an initial error in my calculations that ultimately led to this project. I also thank Louis-Pierre Arguin for his encouragement, generous involvement, and many stimulating discussions. I am grateful to Pascal Maillard for carefully reading an earlier version of the paper, and to Cole Graham for his questions regarding unbounded functionals, which led me to clarify and extend the corresponding results.

\bibliographystyle{plain}
\bibliography{biblio}

\end{document}